\documentclass[twoside]{amsart}

\usepackage{latexsym}
\usepackage{mathrsfs}
\usepackage{amssymb}
\usepackage{amsmath}
\usepackage{amsfonts}
\usepackage{hyperref}
\usepackage{booktabs}
\usepackage{multirow}
\usepackage{subfigure}
\usepackage{floatrow}
\usepackage{algorithmic}
\usepackage{algorithm}
\usepackage{amsthm,array}
\usepackage{array}
\usepackage{mathtools}
\usepackage{tikz}
\usepackage{graphicx}

\font\smc=cmcsc10

\def\br{{\mathbb{R}}}

\def\bz{{\mathbb{Z}}}

\def\br{\mathbb R}

\def\id{{\text{\rm Id}}}

\def\gdeg{G\text{\rm -Deg\,}}

\def\norm#1.{\left\|{#1}\right\|}

\def\vs{\vskip.3cm}
\def\nk#1.#2.{  \left|\begin{array}{c} #1\\ #2\end{array}\right|}

\def\elra{\hbox to 2in{\rightarrowfill}}
\def\ella{\hbox to 2in{\leftarrowfill}}
\def\hrf{\hbox to 2in{\hrulefill}}
\def\hdotfill{\leaders\hbox to 1em{\hss .\hss}\hfill}

\def\vp{\varphi}

\def\id{\text{\rm Id\,}}

\def\ve{\varepsilon}

\def\abv#1|#2.{\scriptstyle \Big[\begin{array}{c} #1 \\ #2 \end{array}\Big]}

\def\dim{\text{\rm dim\,}}

\def\a{\alpha}

\def\s1deg{S^1\text{\rm -Deg\,}}
\def\ker{\text{\rm Ker\,}}

  \newcommand{\amal}[5]{#1\prescript{#4}{}\times^{#5}_{#3}#2}
  \newcommand{\amall}[5]{#1\prescript{#4}{#5}\times^{}_{#3}#2}
  \newcommand{\eqdeg}[1]{#1\mbox{\rm -}\deg}

\def\vs{\vskip.4cm}

\newtheorem{theorem}{Theorem}[section]
\newtheorem{proposition}[theorem]{Proposition}
\newtheorem{lemma}[theorem]{Lemma}

\newtheorem{definition}[theorem]{Definition}
\newtheorem{remark}[theorem]{Remark}
\newtheorem{example}[theorem]{Example}

\newtheorem{remark-definition}[theorem]{Remark and Definition}

\begin{document}
  \title[Bifurcation in Symmetric Reversible FDEs]{Bifurcation of
      Space Periodic Solutions in Symmetric Reversible FDEs}
  \author{Zalman Balanov}
    \address{Department of Mathematical Sciences, University of Texas at Dallas, Richardson, Texas, 75080 USA}
    \email{balanov@utdallas.edu}
  \author{Hao-Pin Wu}
    \address{Department of Mathematical Sciences, University of Texas at Dallas, Richardson, Texas, 75080 USA}
    \email{hxw132130@utdallas.edu}
  \date{January, 2016.}
  \begin{abstract} 
  In this paper, we propose an equivariant degree based method to study bifurcation of periodic
  solutions (of constant period) in symmetric networks of reversible FDEs.
  Such a bifurcation occurs when eigenvalues of linearization move along the imaginary axis
  (without change of stability of the trivial solution and possibly without $1:k$ resonance).
  Physical examples motivating considered settings are related to stationary solutions to
  PDEs with non-local interaction: reversible mixed delay differential equations (MDDEs)
  and integro-differential equations (IDEs). In the case of $S_4$-symmetric networks
  of MDDEs and IDEs, we present exact computations of full equivariant bifurcation invariants.
  Algorithms and computational procedures used in this paper are also included.
  \end{abstract}
  \maketitle
%
  \section{Introduction}\label{sec:intro}
  \paragraph{\bf Subject.}  
  Reversing symmetry is an important subject in natural science (see survey \cite{LambSurvey}).
  Typically, ($R$-symmetric) periodic solutions to one-parameter  ODEs respecting
  the reversing symmetry appear as two-parameter families in which periodic solutions of
  constant period constitute a one-parameter subfamily
  (cf.~Definition \ref{def-reversible} and Theorem 4.3 from \cite{LambSurvey}).
  Local bifurcations of (families of) periodic solutions to parameterized reversible
  systems of ODEs have been studied intensively by many authors
  (see \cite{LambReversibleHopf}, \cite{VanderbauwhedeSurvey} and references therein).
  For example, the reversible codimension-one Hopf bifurcation in parameterized ODEs
  may occur as a result of a collision of  eigenvalues  of the linearization on the
  imaginary axis (the so-called $1:1$-resonance; see \cite{LambReversibleHopf}
  and \cite{VanderbauwhedePRSE}). Higher $1 : k$ resonances, related to the case
  $i \beta_1= k i \beta_2$ ($k \in \mathbb Z$ and $i\beta_1$,
  $i\beta_2$ are the eigenvalues at the moment of the bifurcation) were studied 
  in \cite{Arnol'd} and \cite{Arnol'dSevryuk}.
  In contrast to the reversible Hopf bifurcation scenarios, in this case, the purely imaginary  
  eigenvalues move along the imaginary axis before and after the resonance.
  It should be pointed out that the eigenvalues moving along the imaginary
  axis may give rise to a bifurcation of periodic solutions of
  {\it constant} period even without any resonance (see \cite{KW},
  section 8.6, which is a starting point for our discussion). 
  This bifurcation, considered in systems of functional differential equations
  respecting a finite group of symmetries, is the main {\it subject} of the present paper. 
  As a by-product, we correct some minor inaccuracies in \cite{KW}, section 8.6. 
  \vs
  \paragraph{\bf General setting and motivating examples.}
  In order to describe the general setting, denote by $C_{2n}$ the Banach space of bounded continuous
  functions from $\mathbb{R}$ to $\mathbb R^{2n}:= \mathbb R^n \oplus \mathbb R^n$
  equipped with the sup-norm. For any $\varphi \in C_{2n}$ and $x \in \mathbb R$,
  let $\varphi_x \in  C_{2n}$ be a function defined by $\varphi_x(s):= \varphi(x+s)$
  for $s \in \mathbb R$.
  Assume that $f:\mathbb{R}\times C_{2n}\rightarrow \mathbb R^{2n}$ is a continuous map and
  consider the following  parameterized by $\alpha \in \br$ family of functional differential equations
  \begin{align}
    \label{eq:fde-initial}
    \frac{du}{dx}(x)=f(\alpha,u_x).
  \end{align}
  \vs
  \begin{definition}\label{def-reversible}\rm
    System \eqref{eq:fde-initial} is said to be {\it reversible}, if 
    \begin{align}\label{eq:reversible-system}
      f(\alpha, R\varphi(-s)) = -Rf(\alpha,\varphi(s)) \quad \text{ for all} \;
      \varphi \in C_{2n}, s \in \mathbb R, \alpha \in \mathbb R,  
    \end{align}
    where $R:\mathbb R^n \oplus \mathbb R^n  \to \mathbb R^n \oplus \mathbb R^n$ is given by 
    \begin{align}\label{eq:2.2_def_R}
      R:=\begin{bmatrix}
        \id & 0 \\
        0 & -\id
      \end{bmatrix}.
    \end{align}
    We will also say that \eqref{eq:fde-initial} is {\it $R$-symmetric}.
  \end{definition}  
  \vs
  Let us present few examples of reversible systems \eqref{eq:fde-initial}
  \vs
  \begin{example}[Reversible ODEs]\label{ex:1.1}\rm 
    Let $h:\mathbb R\oplus\mathbb R^{2n}\to\mathbb R^{2n}$
    be a continuous function satisfying  $h(\alpha,Ru) = - Rh(\alpha,u)$ for all
    $(\alpha,u) \in \mathbb R \oplus \mathbb R^{2n}$. Then the system of ODEs
    \[
    \frac{du}{dx}(x)=h(\alpha,u(x))
    \]
    is a particular case of reversible system \eqref{eq:fde-initial}, where 
    $f:\mathbb{R}\times C_{2n}\rightarrow\mathbb R^{2n}$
    is given by $f(\alpha,\varphi):= h(\alpha,\varphi(0))$.
    Clearly, the second order system of ODEs
    \begin{equation}
      \label{eq:reverse-ODE}
     \ddot v(x) =g(\alpha, v(x)),
    \end{equation}
    where $g:\mathbb{R}\times \br^n\rightarrow\mathbb R^{n}$ is a continuous function, can be considered
    as a prototypal example for (parameterized families of) reversible ODEs
    (here $u = (v, \dot v) \in \mathbb R^{2n})$.
  \end{example}  
  \vs
  Another example of reversible system \eqref{eq:fde-initial} generalizing
  \eqref{eq:reverse-ODE}, is the following mixed
  delay differential equation with both positive and negative (i.e., advanced argument) delays.
  \vs 
  \begin{example}[Reversible MDDEs]\label{ex:1.2-1}\rm
    Assume again that $g: \br\times \br^n\rightarrow\mathbb R^{n}$ is continuous, then the equation
    \begin{align}
      \label{eq:dde_system}
      \ddot v(x)=g(\alpha, v(x))+a(v(x-\alpha)+v(x+\alpha)), \quad a\in \br,
    \end{align}
    is a particular case of system \eqref{eq:fde-initial} with $u=(v, \dot v)$ and 
    \begin{align*}
      f(\alpha,v_1,v_2)=(v_2,g(\alpha, v_1))+(0,a(v_1(\alpha)+v_1(-\alpha))),\quad  (v_1,v_2)\in C_{2n}.
    \end{align*}
  \end{example}
  \vs
  Another generalization of \eqref{eq:reverse-ODE} is the following system of  integro-differential equations.
  \vs
  \begin{example}[Reversible IDEs]\label{ex:1.2-2}\rm
    Let $g:\br\times\br^n\rightarrow\mathbb R^{n}$ and
    $k:\br\times \br^n\to \br$ be continuous.
    Put $k_\alpha(y) := k(\alpha,y)$, $y\in \br^n$,
    and assume that for every $\alpha\in \br$,
    $k_\alpha$ is even and has a compact support or is
    a rapidly decreasing function.
    Then, the system of integro-differential equations:
    \begin{equation}\label{eq:ide_system}
      \ddot v(x)=g(\alpha, v(x))+a\int_{-\infty}^\infty{v(x-y)k_\alpha(y)\,dy}, \quad a\in \br,
    \end{equation}
    is a reversible system of type \eqref{eq:fde-initial}.
  \end{example}
  \vs
  Note that by replacing $x$ by $t$ in Examples \ref{ex:1.2-1} and \ref{ex:1.2-2},
  one obtains {\it  time-reversible} FDEs.
  However, such systems involve using the information from the future by
  ``traveling back in time'', which is difficult to justify from
  a commonsensical viewpoint. Therefore, in the present paper,
  we discuss only (symmetric networks of) 
  space-reversible FDEs (one can think of equations governing steady-state
  solutions to PDEs, cf.~\cite{LambSurvey} and references therein).
  Also, the natural source of space-reversibility is related to non-local interaction.
  A prototypal example for equation of type \eqref{eq:dde_system}
  is related to diffusion process with two scales (the local diffusion is
  modeled by the continuous Laplacian while the non-local diffusion is
  related to the discrete Laplacian). For the equations close in spirit to
  \eqref{eq:ide_system}, we refer to \cite{schrodinger}.
  \vs
  In this paper, we are focused on $\Gamma$-symmetric reversible systems of FDEs,
  where $\Gamma$ is a finite group. In such systems, symmetries may come from
  $\Gamma$-symmetrically  coupled networks. As an illustrative example of
  such symmetries, we consider the octahedral group $\Gamma=S_4$.
  \vs
  \paragraph{\bf Method.} The standard way to study local
  bifurcations in equivariant/reversible systems
  is rooted in the singularity theory: assuming the system to
  satisfy several smoothness and genericity conditions 
  around the bifurcation point, one can combine the equivariant/reversible 
  normal form classification with Center Manifold Theorem/Lyapunov-Schmidt Reduction.
  For a detailed exposition of this concept and related techniques, 
  we refer to \cite{GolSchSt,GS,LambReversibleHopf,HLR,LR}.
  \vs
  During the last twenty five years the equivariant degree theory emerged in non-linear analysis
  (for the detailed exposition of this theory, we refer to the monographs
  \cite{AED, IV-B,KB} and survey \cite{survey}), providing alternative
  methods to study bifurcation in reversible systems.
  In short, the equivariant degree is a topological tool allowing ``counting'' orbits of solutions
  to symmetric equations in the same way as the usual Brouwer degree does, but according to
  their symmetry properties. In particular, the equivariant degree theory has all the needed features 
  allowing its application in non-smooth/non-generic settings related to equivariant
  dynamical systems having, in general, infinite dimensional phase spaces.
  \vs
  In the present paper, to establish the abstract results on the existence,
  multiplicity and symmetric properties of bifurcating branches of periodic solutions,
  we use the  $(\Gamma \times O(2))$-equivariant degree without free parameters,
  where $O(2)$ is related to the reversing symmetry while $\Gamma$ reflects the
  symmetric character of the coupling in the corresponding network. We also present
  concrete examples related to $S_4$-symmetric coupling, for which the equivariant
  bifurcation invariant is fully evaluated.
  In order to achieve these computational goals,
  we developed (using some important results obtained in \cite{DKY}) several  
  new group-theoretical computational algorithms,
  which were implemented in the specially created  G.A.P.~program. 
  \vs
  \paragraph{\bf Overview.} After the Introduction, the paper is organized as follows.
  In Section \ref{sec:pre},  we recall the standard equivariant jargon,
  provide isotypical decompositions of functional spaces naturally associated to
  periodic solutions to system \eqref{eq:reversible-system} and outline the
  axiomatic approach to the equivariant degree without free parameters.
  In Section \ref{sec:assu_and_reform}, we reformulate system \eqref{eq:reversible-system}
  as an equivariant fixed-point problem in an appropriate Sobolev space.
  In Section \ref{sec:Gdeg_method_for_bif}, we apply the equivariant degree method to prove
  our main abstract result (see Theorem \ref{thm:suff_cond} and formula
  \eqref{coro:deg_formula}) on the occurrence, multiplicity and symmetric
  properties of bifurcation branches of $2\pi$-periodic solutions to equivariant system
  \eqref{eq:reversible-system}.
  Applications of Theorem \ref{thm:suff_cond} to
  networks of oscillators of type \eqref{eq:dde_system} and \eqref{eq:ide_system}
  coupled in a cube-symmetric fashion, are given in the fifth section
  (cf.~Theorems \ref{thm:dde-num} and \ref{thm:ide-num}).
  Finally, in Appendix, we explain two main computational routines required for symbolic exact evaluation  the
  $(\Gamma\times O(2))$-equivariant degree. Namely, 
   the routine  used  for finding all the conjugacy classes of $\Gamma\times O(2)$
  (see Appendix \ref{appendix:a}),  and the one used  for computing the multiplication
  of two generators in the Burnside ring $A(\Gamma\times O(2))$ -- the range of valued of the
  $(\Gamma \times O(2))$-equivariant degree (see Appendix \ref{appendix:b}).
  \vs 
  \paragraph{\bf Acknowledgement.} The authors were supported by National Science Foundation grant DMS-1413223.
  We are thankful to W.~Krawcewicz for stating the problem and important discussions related to the applications
  of the equivariant degree method in studying symmetric reversible systems.
  We are also thankful to D.~Rachinskii and L.~Kalachev for discussions on applied aspects
  of the results obtained in this paper.

\vs

\section{Preliminaries}
  \label{sec:pre}
  In this section, we review basic terminology and results from equivariant
  topology and representation theory as well as recall the axiomatic approach to the equivariant
  degree without free parameters. In addition,
  we provide a description of subgroups and their conjugacy classes in a direct products of groups.
  \vs
  \subsection{Equivariant Jargon}
  \label{subsec:jargon}
  Let $\mathcal G$ be a group. We will use the notation 
  $H \leq \mathcal G$ to indicate that $H$ is a subgroup of $\mathcal G$.
  For $H\le \mathcal G$ we  denote by $N_{\mathcal G}(H)$ the {\it normalizer} of $H$ in $\mathcal G$, 
  $W_{\mathcal G} (H) :=  N_{\mathcal G}(H)/H$ the {\it Weyl group} of $H$ in $\mathcal G$ and by $(H)$
  the {\it conjugacy class} of $H$ in $\mathcal G$ (we will omit the subscript ``$\mathcal G$''
  if the ambient group is clear from the context).
  In the case $\mathcal G$ is a compact Lie group,
  we will always assume that all the considered subgroups $H\le \mathcal G$ are closed.
  \vs
  The set $\Phi(\mathcal G)$ of all conjugacy classes of subgroups in $\mathcal G$ can be naturally
  equipped with the partial order: $(H)\leq (K)$ if and only if
  $gHg^{-1}\leq K$ for some $g\in \mathcal G$.
  \vs
  In what follows, $G$ will stand for a compact Lie group. For any integer $n\geq 0$, put
  $\Phi_n(G):=\{(H)\in\Phi(G) : \dim W(H)=n\}$.
  \vs
  Suppose $X$ is a $G$-space and $x \in X$.  Denote by
  $G_x:=\{g\in G: gx=x\}$ the {\em isotropy} of $x$, by $G(x):=\{gx: g\in G\}$
  the {\it orbit} of $x$, and by $X/G$ the {\it orbit space}.
  For any isotropy $G_x$, call $(G_x)$ the {\it orbit type} of $x$ and put
  $\Phi(G;X):=\{(H)\in\Phi(G): H=G_x\mbox{ for some }x\in X\}$ and  
  $\Phi_n(G;X) := \Phi(G;X) \cap \Phi_n(G)$. 
  Also, for any $H \leq G$, denote by $X^H:=\{x\in X: G_x\geq H\}$ the set of $H$-fixed points and
  put $X^{(H)}:=\{x\in X:(G_x)\geq(H)\}$, $X_{H}:=\{x\in X: G_x =H\}$,
  $X_{(H)}:= \{x\in X:(G_x) =(H)\}$. It is well-known
  that $W(H)$ acts on $X^H$ and this action is free on $X_H$. 
  \vs
  Given two subgroups $H \leq K \leq G$, define $N_G(H,K):=\{g\in G:  gHg^{-1}\leq K\}$.
  Obviously, $N_G(H,K)$ is a left $N_G(K)$-space (also, it is a right $N_G(H)$-space).
  Moreover (see, \cite{AED}, Proposition 2.52), if $\dim W_G(H) = \dim W_G(K)$,
  then $n_G(H,K) :=\left\lvert N_G(H,K)/N_G(K)\right\rvert$ is finite (cf.~\cite{IG,KB}). 
  We will also omit the subscript ``$G$'' when the group $G$ is clear from the context. 
  Recall that the number $n(H,K)$ coincides with the number of conjugate copies of
  $K$ in $G$ which contains $H$ (cf.~\cite{IG,KB}).
  \vs
  Suppose $Y$ is another $G$-space. A continuous map
  $f:X\rightarrow Y$ is called {\it $G$-equivariant} if $f(gx)=gf(x)$
  for all $x\in X$ and $g\in G$. For any $H\leq G$ and
  equivariant map $f:X\rightarrow Y$, the map
  $f^H:X^H\rightarrow Y^H$, with $f^H:f|_{X^H}$,
  is well-defined and $W(H)$-equivariant.

\vs 
  Finally, given two Banach spaces $E_1$ and $E_2$ and an open bounded subset
  $\Omega\subset E_1$,  a continuous  map $f : E_1\rightarrow E_2$ is
  called {\em $\Omega$-admissible} if $f(x) \neq 0$ for all  $x \in \partial\Omega$. In this case,
  $(f,\Omega)$ is called an {\it admissible pair}. Denote by $\mathcal M(E_1,E_2)$ the set 
  of all $\Omega$-admissible pairs.
  \vs
  
  We refer to \cite{Bre,tD,AED} for additional equivariant topology background used in the present paper.
  \vs
  \subsection{Subgroups in Direct Product of Groups $\mathscr G_1\times
  \mathscr G_2$}\label{subsec:direct_product_subgroup}
  Given two groups $\mathscr G_1$ and $\mathscr G_2$,
  consider  the product group and $\mathscr G_1\times \mathscr G_2$ and
  define the projection homomorphisms:
  \begin{align*}
    \pi _{1}&:\mathscr G_1\times\mathscr G_2\to \mathscr G_1,\quad\pi_1(g_1,g_2)=g_1;\\
    \pi _{2}&:\mathscr G_1\times\mathscr G_2\to \mathscr G_2,\quad\pi_2(g_1,g_2)=g_2.
  \end{align*}
  The following well-known result (see \cite{DKY} for more details),
  which is rooted in Goursat's Lemma (cf.~\cite{Goursat}),
  provides the desired description of subgroups $\mathscr H$ of the
  product group $\mathscr G_1\times \mathscr G_2$.
  \begin{theorem}\label{th:product-1}
    Let  $\mathscr H$ be a subgroup of the product group
    $\mathscr G_1\times \mathscr G_2$. Put $H:=\pi_1(\mathscr H)$
    and $K:=\pi_2(\mathscr H)$. Then, there exist a group $L$ and
    two epimorphisms $\vp :H\rightarrow L$ and $\psi :K\rightarrow L$,
    such that 
    \begin{equation}\label{eq:product-rep}
      \mathscr H=\{(h,k)\in H\times K: \vp(h)=\psi(k)\},
    \end{equation}
    (see Figure \ref{fig:subg_com_diagram1}). In this case, we will use the notation
    \begin{align}
      \label{eq:subg_notation}
      \mathscr H=:\amal{H}{K}{L}{\varphi}{\psi}.
    \end{align}
  \end{theorem}
  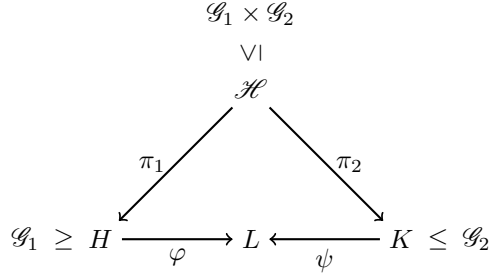
\begin{figure}[h]
    \centering
    \begin{tikzpicture}
      \node (G) at (0,0) {$\mathscr G_1\times \mathscr G_2$};
      \node (S) at (0,-1) {$\mathscr H$};
      \node (HH) at (-3,-3) {$\mathscr G_1$};
      \node (KK) at (3,-3) {$\mathscr G_2$};
      \node (H) at (-2,-3) {$H$};
      \node (K) at (2,-3) {$K$};
      \node (L) at (0,-3) {$L$};
      \path (G)--(S) node[pos=0.5,rotate=90] {$\le$};
      \draw[->, thick] (S)--(H) node[pos=0.5, left] {$\pi_1$};
      \draw[->, thick] (S)--(K) node[pos=0.5, right] {$\pi_2$};
      \path (HH)--(H) node[pos=0.5] {$\ge$};
      \path (KK)--(K) node[pos=0.5] {$\le$};
      \draw[->, thick] (H)--(L) node[pos=0.5, below] {$\varphi$};
      \draw[->, thick] (K)--(L) node[pos=0.5, below] {$\psi$};
    \end{tikzpicture}
    \caption{ Subgroup  $\mathscr H\le \mathscr G_1\times \mathscr G_2$}
    \label{fig:subg_com_diagram1}
  \end{figure}
  In order to describe the conjugacy classes of subgroups
  of $\mathscr G_1\times \mathscr G_2$, one needs the following statement (see \cite{DKY}).
  \vs
  \begin{proposition}\label{prop:conj-classes}
    Let $\mathcal G_1$ and $\mathcal G_2$ be two groups.   
    Two subgroups $H{^\varphi \times _L^\psi}K, H'{^{\varphi'} \times _L^{\psi'}}K'$
    of $\mathcal G_1 \times \mathcal G_2$  are conjugate if and only if there exist 
    $(a,b)\in \mathcal G_1\times \mathcal G_2$ such that the inner
    automorphisms $a_.: \mathcal G_1\to \mathcal G_1$ and $b_.:\mathcal G_2\to \mathcal G_2$ given by
    \begin{align*}
      \a_.g_1=ag_1a^{-1}, \quad b_.g_2=bg_2b^{-1},\quad g_1\in \mathcal G_1, \; g_2\in \mathcal G_2,
    \end{align*}
    satisfy the properties: $H'=a_.H$, $K'=b_.K$ and $\vp=\vp'\circ a_.$,
    $\psi=\psi'\circ b_.$. In other words, the diagram shown in Figure \ref{fig:conj-diag} commutes. 
    \begin{figure}[h]
      \centering
      \begin{tikzpicture}[scale=.7]
        \node (H) at (-3.464,2) {$H$};
        \node (K) at (3.464,2) {$K$};
        \node (HH) at (-3.464,-2) {$H'$};
        \node (KK) at (3.464,-2) {$K'$};
        \node (L) at (0,0) {$L$};
        \draw[->, thick] (H)--(HH) node[pos=0.5, left] {$a_.$};
        \draw[->, thick] (K)--(KK) node[pos=0.5, right] {$b_.$};
        \draw[->, thick] (H)--(L) node[pos=0.5, above] {$\varphi$};
        \draw[->, thick] (K)--(L) node[pos=0.5, above] {$\psi$};
        \draw[->, thick] (HH)--(L) node[pos=0.5, below] {$\varphi'$};
        \draw[->, thick] (KK)--(L) node[pos=0.5, below] {$\psi'$};
      \end{tikzpicture}
      \caption{Conjugacy classes of subgroups in direct product}\label{fig:conj-diag}
    \end{figure}
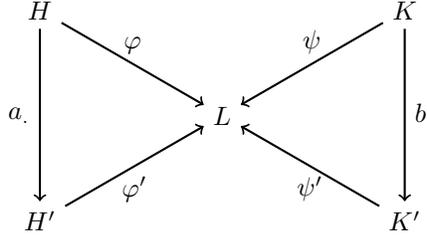
  \end{proposition}
%
  \subsection{$(\Gamma\times O(2))$-representations}
  \label{subsec:g_rep}
  Let $\Gamma$ be a {\em finite} group and $G:=\Gamma \times O(2)$. In this subsection, we
  describe $(\Gamma\times O(2))$-representations on function spaces relevant for studying periodic
  solutions to system \eqref{eq:fde-initial}.
  We also discuss the related $G$-isotypical decompositions.
  \vs
  \paragraph{\bf (a) $G$-Representations on Function Spaces.} Denote by $V^1$ and $V^2$ two
  identical copies of an $n$-dimensional orthogonal
  $\Gamma$-representation. Put  $V:=V^1\oplus V^2$ and consider  the Sobolev spaces
  $\mathscr W^k:=H^1(S^1;V^k)$ for $k=1,2$ (here, we use the standard identification
  $S^1=\mathbb R/2\pi\mathbb Z$). Clearly, the space $\mathscr W^k$ 
  is a Hilbert $G$-representation with the $G$-action defined  as follows: given $w^k\in \mathscr W^k$,
  \begin{align}
    \label{eq:2.3_Gaction_on_W_1}
    ((\gamma,1)w^k)(\cdot)&=\gamma w^k(\cdot)\\
    \label{eq:2.3_Gaction_on_W_2}
    ((1_\Gamma,e^{i\theta})w^k)(\cdot)&=(\phi_\theta w^k)(\cdot)\\
    \label{eq:2.3_Gaction_on_W_3}
    ((1_\Gamma,\kappa)w^k)(\cdot)&=(-1)^{k-1}(Tw^k)(\cdot)
  \end{align}
  where $\gamma\in \Gamma$ and $e^{i\theta}$, $\kappa\in O(2)$, and 
  \begin{align}
    \label{eq:2.2_def_phi}
    &(\phi_\theta w^k)(t)=w^k(t+\theta)\\
    \label{eq:2.2_def_T}
    &(Tw^k)(t)=w^k(-t).
  \end{align}
  Put $\mathscr W:=\mathscr W^1\oplus \mathscr W^2$.
  \vs
  In a similar way, we have  the Banach $G$-representation $C:=C(S^1;V)$,
  where the $G$-action is given by the same formulas 
  \eqref{eq:2.3_Gaction_on_W_1}--\eqref{eq:2.2_def_T},
  and the Hilbert $G$-representation $E:=L^2(S^1;V)$,
  with the $G$-action slightly modified, i.e., instead of
  \eqref{eq:2.3_Gaction_on_W_3} we define
  \begin{align}
    \label{eq:g_action4}
    ((1_\Gamma,\kappa)w^k)(\cdot)&=(-1)^k(Tw^k)(\cdot)\quad (k=1,2).
  \end{align}
  \vs
  \paragraph {\bf (b) Isotypical Decompositions.}
  We assume that $\{\mathcal V_j\}_{j=0}^r$ is the complete list of all
  irreducible $\Gamma$-representations, where $\mathcal V_0$
  stands for the trivial representation (recall, $\Gamma$ is finite).
  Since $V^1$ and $V^2$ are equivalent $\Gamma$-representations,
  the $\Gamma$-isotypical decompositions of $V^k$ are both
  \begin{align}
    \label{eq:2.3_V_decomp}
    V^k=\bigoplus_{j=0}^{r}{V_j}
  \end{align}
  for $k=1,2$, where $V_j$ is modeled on a $\Gamma$-irreducible representation $\mathcal{V}_j$.
  Therefore, the $\Gamma$-isotypical decomposition of $\mathscr W^k$ is 
  (see \eqref{eq:2.3_Gaction_on_W_1}),
  \begin{align}
    \label{eq:2.3_Wk_decomp}
    \mathscr  W^k\simeq\bigoplus_{j=0}^r{\mathscr W_j^k},
  \end{align}
  where
  \begin{align}
    \label{eq:2.3_Wkj_def}
    \mathscr W_j^k:=H^1(S^1;V_j).
  \end{align}
  Next, consider the Fourier mode decomposition of $\mathscr W_j^k$ 
  \begin{align}
    \label{eq:2.3_Wkj_decomp}
    \mathscr W_j^k\simeq\mathscr W_{j,0}^k\oplus\overline{\bigoplus_{l=1}^\infty{\mathscr W_{j,l}^k}},
  \end{align}
  where $\mathscr W_{j,0}^k=V_j$ and for $l>0$, 
  \begin{align*}
    \mathscr W_{j,l}^k:=\{\cos(lt)a + \sin(lt)b: a,\,b \in V_j\}
  \end{align*}
  (the closure is taken with respect to the Sobolev norm). 
  Since the functions in  $\mathscr W_{j,0}^k$ are constant,
  the component $\mathscr W_{j,0}^k$ can be identified with $V_j$.
  On the other hand, for $l>0$, the component $\mathscr W_{j,l}^k$
  can be identified with the complexification 
  \begin{align*}
    (V_j)^\mathbb{C}:=V_j \oplus iV_j=\{a+ib: a,\, b\in V_j\}
  \end{align*}
  using the following relation:
  \begin{align}
    \label{eq:Wjlk}
    \cos(l\cdot)a+\sin(l\cdot)b\mapsto a-ib\quad a,\,b\in V_j.
  \end{align}
  Define the following action of $G$ on the space  $(V_j)^\mathbb{C}$:
  \begin{align}
    \label{eq:2.3_Gaction_on_Vjc_1}
    (\gamma,1)z&=\gamma z\\
    \label{eq:2.3_Gaction_on_Vjc_2}
    (1_\Gamma,e^{i\theta})z&=e^{il\theta}\cdot z\\
    \label{eq:2.3_Gaction_on_Vjc_3}
    (1_\Gamma,\kappa)z&=(-1)^{k-1}\overline{z}.
  \end{align}
  where $z:=a+ib$, $a$, $b\in V_j$ and `$\cdot$' denotes the usual complex multiplication. 
  Then, one can easily verify that the identification \eqref{eq:Wjlk} is
  $G$-equivariant.  Indeed, for example, we have for $w(t)=\cos(lt)a+\sin(lt)b$
  \begin{align*}
    (1_\Gamma,e^{i\theta})w&=\cos (lt)(\cos( l\theta) a+\sin(l\theta) b)+
      \sin (lt)(\cos(l\theta)b-\sin(l\theta)b)\\
    &\mapsto (\cos( l\theta) a+\sin(l\theta) b)-i(\cos(l\theta)b-\sin(l\theta)b)\\
    &=e^{il\theta}(a-ib).
  \end{align*}
  \begin{remark}\label{rem:2.3_equi_comp}\rm
    Notice that $\mathscr W_{j,0}^1$ and $\mathscr W_{j,0}^2$ are not equivalent $G$-representations.
    However, $\mathscr W_{j,l}^1$ and $\mathscr W_{j,l}^2$ (for $l>0$) are equivalent
    $G$-representations. Indeed, define the mapping
    $\eta:\mathscr W_{j,l}^1\rightarrow \mathscr W_{j,l}^2$ by 
    \begin{align}
      \eta(x+iy)=i(x+iy)=-y+ix
    \end{align}
    Clearly, $\eta$ is $\Gamma\times SO(2)$-equivariant.
    However, notice that the $\kappa$-actions also commute with $\eta$,
    which is shown in Figure \ref{fig:eta}.
    \begin{figure}[h]
      \centering
      \begin{tikzpicture}
        \node (k1) at (-2,2) {$\mathscr W_{j,l}^1$};
        \node (k2) at (2,2) {$\mathscr W_{j,l}^2$};
        \node (11) at (-2,1) {$x+iy$};
        \node (12) at (-2,-1) {$x-iy$};
        \node (21) at (2,1) {$-y+ix$};
        \node (22) at (2,-1) {$y+ix$};
        \draw[dashed] (0,2.5)--(0,-1.5);
        \draw[->] (11)--(12) node[pos=0.5, left] {$\kappa$};
        \draw[->] (21)--(22) node[pos=0.5, right] {$\kappa$};
        \draw[->] (11)--(21) node[pos=0.3, above] {$\eta$};
        \draw[->] (12)--(22) node[pos=0.3, below] {$\eta$};
      \end{tikzpicture}
      \caption{Isomorphism $\eta$.}\label{fig:eta}
    \end{figure}
  \end{remark}
  \vs
  Combining \eqref{eq:2.3_V_decomp}--\eqref{eq:2.3_Gaction_on_Vjc_3} 
  with Remark \ref{rem:2.3_equi_comp} yields the following
  $G$-isotypical decomposition of $\mathscr W$:
  \begin{align}
    \label{eq:2.2_W_isodecomp}
    \mathscr W=\bigoplus_{k=1}^2\bigoplus_{j=1}^r{\mathscr W_{j,0}^k}\oplus
        \overline{\bigoplus_{l=1}^\infty\bigoplus_{j=1}^r{\mathscr W_{j,l}}},
  \end{align}
  where 
  \begin{equation}\label{eq:Wjl-decomposition}
  \mathscr W_{j,l}:=\mathscr W_{j,l}^1\oplus \mathscr W_{j,l}^2.
  \end{equation} 
  \vs 
  In the sequel, we will also use
  the following notation:
  \begin{align}
    \label{eq:fourier_mode}
   \mathscr W_l:=\bigoplus_{j=1}^r{\mathscr W_{j,l}}.
  \end{align}
\vs

For additional information about the representation theory, we refer to \cite{BtD,tD,AED}.
\vs

  \subsection{Admissible $G$-pairs  and Burnside Ring  $A(G)$}
  \label{subsec:burnside}
Put $\mathcal{M}:=\bigcup_V \mathcal{M}(V,V)$,
  where $V$ runs over the set of Euclidean spaces
  (see Subsection \ref{subsec:jargon}). The Brouwer degree is the function $\deg : \mathcal M\to \bz$, which, for a given admissible pair $(f,\Omega)$, 
provides an algebraic count of zeroes of $f$ in the domain $\Omega$. Its standard properties (existence, additivity, homotopy and normalization) can be used as axioms (see, for example,  \cite{KW}). 
  In symmetric settings, when dealing with the equivariant maps $f$, zeroes of $f$ usually come in orbits and in order to provide a similar algebraic count of these orbits one also needs to take into account their symmetry properties (i.e., their orbit types). An appropriate tool in these settings  such a degree is the so-called equivariant degree without parameter (see, for example,
  \cite{KW,survey,AED,KB}).

\vs
Let us briefly recall the properties of the $G$-equivariant degree without parameter (which can actually be used as axioms). An admissible pair 
 $(f,\Omega)\in\mathcal{M}(V,V)$  is called an {\it admissible $G$-pair} if $\Omega$ is $G$-invariant and $f$ is $G$-equivariant.
 Denote by $\mathcal{M}^G(V)$ the set of all admissible $G$-pairs in $V$ and 
 put
  \begin{align}
    \mathcal{M}^G:=\bigcup_V{\mathcal{M}^G(V)},
  \end{align}
 where the union is taken over all orthogonal $G$-representations $V$. 
 The collection $\mathcal{M}^G$ replaces $\mathcal{M}$ in the equivariant setting. A continuous map  $h : [0,1] \times V \to V$ 
 is called an {\it $\Omega$-admissible $G$-homotopy} if $(h(t, \cdot), \Omega) \in   \mathcal{M}^G(V) $ for any $t \in [0,1]$. 
\vs 
In the $G$-equivariant degree, the ring $\mathbb{Z}$ is replaced by the so-called {\it Burnside ring} $A(G)$. To be more specific, 
$A(G):=\mathbb{Z}[\Phi_0(G)]$ is a $\bz$-module, i.e.,
it is   the free $\mathbb Z$-module generated by $(H)\in\Phi_0(G)$
  (see Subsection \ref{subsec:jargon}). Notice that elements of $A(G)$ can be written as finite sums 
  \[ 
  n_{1} (H_1)+n_2(H_2) +\dots +n_m(H_m)= \sum_{k=1}^m{n_{k}(H_k)}, \quad n_{k} \in \mathbb Z, \; (H_k) \in \Phi_0(G).\]
  Occasionally, it will be convenient to write these elements as 
  \[
  \sum_{(H)\in \Phi_0(G)} n_{H}(H), \quad n_{H} \in \mathbb Z,
  \]
  with finitely many $n_H\not=0$.
  To define the ring multiplication ``$\cdot$'' in $A(G)$, take $(H)$, $(K) \in \Phi_0(G)$ and observe that $G$ acts diagonally on $G/H \times G/K$ with finitely many orbit types. Consider such an orbit type  $(L) \in \Phi_0(G)$. As is well-known (see, for example, \cite{Bre,AED}),  $(G/H \times G/K)_{(L)}/G$ is finite. Put
  \begin{align}\label{eq:Burnside-geom}
    (H)\cdot(K):=\sum_{(L)\in\Phi_0(G)}{m_L(H,K)}(L),
  \end{align}
where $m_L(H,K):= |(G/H \times G/K)_{(L)}/G|$. Then, the $\mathbb Z$-module $A(G)$ equipped with the above multiplication (extended from generators by distributivity) becomes a ring. Notice that $A(G)$ is a ring with the unity $(G)$, i.e., $(H)\cdot (G)=(H)$ for every $(H)\in \Phi_0(G)$.

\eject
  \subsection{Equivariant degree without parameter}
  \label{subsec:g_basic_degree}
 \paragraph{\bf Axioms.} We follow the axiomatic approach to the definition of the equivariant degree without parameter given in \cite{survey} (for more information on the equivariant degree theory see also \cite{AED,IV-B}) . 

  \renewcommand{\labelenumi}{\rm(G\arabic{enumi})}
\vs

  \begin{theorem}\label{thm:G-deg}
There exists a unique map (called the $G$-equivariant degree without parameter) $\eqdeg{G}:\mathcal{M}^G\rightarrow A(G)$, which for each admissible $G$-pair $(f,\Omega)$, associates an element 
\begin{equation}   \label{eq:2.5_gdeg_sum}
\gdeg(f,\Omega)=\sum_{k=1}^m{n_{k}(H_k)}\in A(G)
\end{equation}

    satisfying the following properties:

    \begin{enumerate}
    \item\label{eqdegprop:existence} {\smc (Existence)} If $\eqdeg{G}(f,\Omega)\neq0$, i.e.,
      $n_{k}\neq 0$ for some $k$ in \eqref{eq:2.5_gdeg_sum}, then
      there exists $x\in\Omega$ such that $f(x)=0$ and $(G_x)\geq(H_k)$.
%
    \item\label{eqdegprop:additivity} {\smc (Additivity)} If $\Omega_1$ and $\Omega_2$ are two disjoint
      open $G$-invariant subsets of $\Omega$ such that
      \begin{align*}
        f^{-1}(0)\cap\Omega \subset \Omega_1\cup\Omega_2,
      \end{align*}
      then
      \begin{align*}
        \eqdeg{G}(f,\Omega)=\eqdeg{G}(f,\Omega_1)+\eqdeg{G}(f,\Omega_2).
      \end{align*}
    \item\label{eqdegprop:homotopy} {\smc (Homotopy)} If $h:[0,1]\times V\rightarrow V$ is an $\Omega$-admissible
      $G$-homotopy, then
      \begin{align*}
        \eqdeg{G}(h_t,\Omega)=\mbox{constant}.
      \end{align*}
    \item\label{eqdegprop:normalization} {\smc (Normalization)} If $\Omega\subset V$ is a $G$-invariant open bounded neighborhood of $0$,
      then
      \begin{align*}
        \eqdeg{G}(\id,\Omega)=1(G).
      \end{align*}
    \item\label{eqdegprop:multiplicity} {\smc (Multiplicativity)} For any $(f_1,\Omega_1),\,(f_2,\Omega_2)\in\mathcal{M}^G$,
      \begin{align*}
        \eqdeg{G}(f_1\times f_2,\Omega_1\times\Omega_2)=\eqdeg{G}(f_1,\Omega_1)\cdot\eqdeg{G}(f_2,\Omega_2)
      \end{align*}
      where ``$\cdot$'' stands for the multiplication in the Burside ring $A_0(G)$.
    \item\label{eqdegprop:suspension}  {\smc (Suspension)} If $W$ is an orthogonal $G$-representation and
      $\mathcal{B}$ is an open bounded invariant neightborhood of $0\in W$,
      then
      \begin{align*}
        \eqdeg{G}(f\times\id_W,\Omega\times\mathcal{B})=\eqdeg{G}(f,\Omega).
      \end{align*}
      
    \item \label{eqdegprop:recurrence} {\smc (Recurrence Formula)} For an admissible $G$-pair $(f,\Omega)$,
      the $G$-degree in \eqref{eq:2.5_gdeg_sum} can be computed using the
      following recurrence formula:
      \begin{align*}
        n_H=\frac{\deg(f^H,\Omega^H)-\sum_{(K)>(H)}{n_Kn(H,K)\left\lvert W(K)\right\rvert}}
            {\left\lvert W(H)\right\rvert},
      \end{align*}
      where $\left\lvert X\right\rvert$ stands for the number of elements
      in $X$ and $\deg(f^H,\Omega^H)$ is the Brouwer degree of the
      map $f^H:=f|_{V^H}$ on the set $\Omega^H\subset V^H$.
    \end{enumerate}
 \end{theorem}

\vs
  \begin{remark}\label{deg-Leray-Schauder}\rm
  Combining Property (G\ref{eqdegprop:suspension})  with the standard (equivariant) Leray-Schauder projection techniques (see, for example, \cite{KW,AED}), one can easily extend the
  equivariant degree without parameter to equivariant compact vector fields on Banach $G$-representations.
  \end{remark}
  \vs
 \paragraph {\bf Basic degrees.}
  Let $V$ be an orthogonal $G$-representation and let $L:V\to V$
  be a $G$-equivariant linear iisomorphism. 
  Keeping in mind  the formula for the Brouwer degree of a linear map, denote by
  $\sigma_-(L)$ the set of all negative real eigenvalues of the 
  operator $L$ and let $E(\mu)$ be the generalized eigenspace of $L$ corresponding to $\mu$
  (which is clearly $G$-invariant). To obtain an effective formula
  for the computation of $\eqdeg{G}(L,B(V))$, where $B(V)$ stands for the unit ball in $V$,
  take the isotypical decomposition
  \begin{align}\label{eq:isotop-deco}
    V = V_0\oplus\dots\oplus V_r,
  \end{align} 
  where $V_i$ is modeled on $\mathcal V_i$, put 
  \begin{align}\label{eq:mu-multiplicity}
    m_i(\mu):=\dim(E(\mu)\cap V_i)/\dim\mathcal V_i, \quad 0 = 1,..., r,
  \end{align}  
  and call $m_i(\mu)$ the {\it isotypical $\mathcal{V}_i$-multiplicity} of $\mu$.
  Then, for any irreducible representation   $\mathcal V_i$, put
  \begin{align}\label{eq:basic-degree}
    \deg_{\mathcal V_i} :=  \eqdeg{G}(-{\rm Id}, B(\mathcal V_i)), 
  \end{align}  
  and call $\deg_{\mathcal V_i}$ the basic $G$-degree corresponding to the irreducible representation $\mathcal V_i$. 
  One can easily prove (see \cite{AED}) that for every  basic degree,
  \begin{equation}\label{eq:basic-square}
      (\deg_{\mathcal V_i}) ^2=    \deg_{\mathcal V_i} \cdot     \deg_{\mathcal V_i} =(G).
  \end{equation}
  In addition, we will also use the convention $a^0=(G)$ for any element $a\in A(G)$.
  \vs
  
  Combining the multiplicativity and homotopy properties
  of the equivariant degree  yields the following statement (cf.~\cite{AED,survey}).
  \begin{proposition}\label{prop;basic-degree} 
    Let $V$ be an orthogonal  $G$-representation with the isotypical decomposition \eqref{eq:isotop-deco} and let $T: V \to V$ be an invertible $G$-equivariant linear operator.
    Then (cf.~\eqref{eq:mu-multiplicity} and \eqref{eq:basic-degree}),
    \begin{equation}\label{eq:basic-degree-formula}
      \eqdeg{G}(T, B) = \prod_{\mu \in \sigma_-(T)} 
      \prod_{i=0}^s (\deg_{\mathcal V_i})^{m_i(\mu)}=\prod_{\mu \in \sigma_-(T)} 
      \prod_{i=0}^s (\deg_{\mathcal V_i})^{\ve_i(\mu)}
    \end{equation}
    where the product is taken in the Burnside ring $A(G)$ and
     \[
    \ve_i(\mu):=\begin{cases} 1 &\text{ if } m_i(\mu) \text{ is odd};\\ 0 &\text{ if } m_i(\mu) \text{ is even}
      \end{cases}.\]
   
  \end{proposition}
\vs

  \section{Assumptions and fixed-point problem reformulation}\label{sec:assu_and_reform}

  \subsection{Assumptions}\label{subsec:hypo}
  In this subsection, we describe the setting in which the bifurcating
  branches of $2\pi$-periodic solutions to \eqref{eq:fde-initial} will be studied.
  Let $\Gamma$ and $V$ be as in Subsection \ref{subsec:g_rep}(a).
  Consider a continuous map $f:\mathbb{R}\times C_{2n}\rightarrow \mathbb R^{2n}$
  and assume that the following conditions are satisfied.
  \vs
  \begin{enumerate}
  \renewcommand{\labelenumi}{(P\arabic{enumi})}

  \item{\label{assu:P1}}{\smc (Branch of Equilibria)}
    $f(\alpha,0)=0$ for any $\alpha\in\mathbb{R}$.
  \item{\label{assu:P2}}{\smc (Regularity)}
     $f$ is continuous, $D_uf(\alpha,0)$ exists for any $\alpha\in\mathbb{R}$ and
     depends continuously on $\alpha$ and
     \begin{align*}
       \lim_{\left\lVert u\right\rVert\rightarrow 0}\sup_{\alpha\in[a,b]}
           {\frac{\left\lVert f(\alpha,u)-D_uf(\alpha,0)u\right\rVert}
           {\left\lVert u\right\rVert}}=0
     \end{align*}
     for any $a,b\in\mathbb R$.
  \item{\label{assu:P3}}{\smc (No Steady-state Bifurcation)}
    $\left.D_uf(\alpha,0))\right\rvert_V:V\rightarrow V$ is invertible for any $\alpha$.
  \end{enumerate}
  
  \vs
  
  To formulate the next condition, consider the linearized system
  \begin{align}
    \label{eq:linearized_fde}
    \frac{du}{dx}(x)=D_uf(\alpha,0)u_x.
  \end{align}
  By substituting $u(x)=e^{\lambda x}v$ with $\lambda\in\mathbb{C}$ and $v\in V$
  into \eqref{eq:linearized_fde},
  we obtain the {\it  characteristic operator}
  $\triangle_{\alpha}(\lambda):V^\mathbb C\rightarrow V^\mathbb C$
  for \eqref{eq:linearized_fde}
  \begin{align}
    \label{eq:cheq}
      \triangle_{\alpha}(\lambda):=\lambda\id_V-D_uf(\alpha,0)\eta_\lambda
  \end{align}
  where $(\eta_\lambda v)(t):=e^{\lambda t}v$ for $v\in V^\mathbb C$. Put
  \begin{align}
    \label{eq:Psi}
    \Psi:=\{(\alpha,0):\mathrm{det}_\mathbb{C}[\Delta_\alpha(ik)]=0\;\mbox{for some } k\in\mathbb{N}\}.
  \end{align}
  \vs 
  Now, we can formulate the following condition.
  \vs
  \begin{enumerate}
  \renewcommand{\labelenumi}{(P\arabic{enumi})}
  \setcounter{enumi}{3}
  \item{\label{assu:P4}}{\smc (Isolated Center)}
    There exists $\alpha_o\in\mathbb R$ such that $(\alpha_o,0)$ is isolated in $\Psi$,
    i.e., there exists an open neighborhood $N\subset \mathbb R\times C_{2n}$ of $(\alpha_o, 0)$
    such that $N\cap\Psi=(\alpha_o,0)$.
\end{enumerate}
\vs
Finally, we assume the following conditions to be satisfied.
\vs
  \begin{enumerate}
  \renewcommand{\labelenumi}{(P\arabic{enumi})}
  \setcounter{enumi}{4}
  \item{\label{assu:P5}}{\smc (Symmetry)}
    $f$ is $\Gamma$-equivariant, i.e.,
    \begin{align*}
      f(\alpha,\gamma u)=\gamma f(\alpha,u)
    \end{align*}
    for any $\gamma\in\Gamma$, $\alpha\in\mathbb{R}$ and $u\in C_{2n}$
    ($\Gamma$ acts trivially on $\mathbb{R}$).

  \item{\label{assu:P6}}{\smc (Reversibility)}
    \begin{align}
      f(\alpha,RTu)=-Rf(\alpha,u)\quad\mbox{(cf.~\eqref{eq:2.2_def_R})}.
    \end{align}
  \end{enumerate}

\vs
  \begin{remark}\label{rmk:reversibility}\rm Let us observe that the condition 
    (P\ref{assu:P6}) imposes  some restrictions on the matrix $D_uf(\alpha,0)$.
    To be more specific, by the chain rule
    \begin{align}
      \label{eq:linearized_P6}
      (D_uf)(\alpha,0)R=-R(D_uf)(\alpha,0).
    \end{align}
Assume that
    \begin{align*}
      (D_uf)(\alpha,0)=\begin{bmatrix}
        A & B \\
        C & D
      \end{bmatrix},
    \end{align*}
    where each block in the matrix is a linear transformation from $C_n$ to $\mathbb{R}^n$ (cf.~Section \ref{sec:intro}).
    Then, the condition\eqref{eq:linearized_P6} implies
    \begin{align*}
      &\begin{bmatrix}
        A & B \\
        C & D
      \end{bmatrix}
      \begin{bmatrix}
        \id & 0 \\
        0 & -\id
      \end{bmatrix}=
      -\begin{bmatrix}
        \id & 0 \\
        0 & -\id
      \end{bmatrix}
      \begin{bmatrix}
        A & B \\
        C & D
      \end{bmatrix}\\
      \Rightarrow\quad&
      \begin{bmatrix}
        A & -B \\
        C & -D
      \end{bmatrix}=
      \begin{bmatrix}
        -A & -B \\
        C & D
      \end{bmatrix},
    \end{align*}
    which means
    \begin{align}
      (D_uf)(\alpha,0)=\begin{bmatrix}
        0 & B \\
        C & 0
      \end{bmatrix}.
    \end{align}
  \end{remark}
\vs
  In the next subsection, assuming conditions {\rm (P1)}--{\rm (P6)} 
  to be satisfied, we will reformulate
  \eqref{eq:fde-initial} as a $G$-equivariant fixed-point problem in the Sobolev space $\mathscr W$ (cf.~Subsection \ref{subsec:g_rep}(a)).
\vs 

  \subsection{$G$-Equivariant Operator Reformulation of   \eqref{eq:fde-initial}  in Functional Spaces}
  \label{subsec:reform}
  Let $G := \Gamma \times O(2)$ and let $\mathscr W$, $C$ and $E$ be as in Subsection \ref{subsec:g_rep}(a). Consider the linear operator $L:\mathscr W\rightarrow E$ given by
  \begin{align}
      Lu=\dfrac{du}{dx}
  \end{align}
  and the operator $j:\mathscr W\rightarrow C$ being the natural (compact) embedding of $\mathscr W$ into $C$.
  In addition, let $N_f:\mathbb{R}\times C\rightarrow E$ be the Nemytsky operator induced by $f$:
  \begin{align}
      N_f(\alpha,u)(x):=f(\alpha,u_x).
  \end{align}
  Then, reformulate \eqref{eq:fde-initial} as
  \begin{align}
    \label{eq:3.2_op_reform}
    Lu=N_f(\alpha,j(u)),\quad u\in \mathscr W.
  \end{align}
\vs
  \begin{remark} \rm
Since the system  \eqref{eq:fde-initial} is reversible and $\Gamma$-symmetric (see conditions (P5) and (P6)), the operators $L$, $j$ and $N_f$ are $G$-equivariant.
    Indeed, the $G$-equivariance of   $j$ is obvious. For the operator  $L:\mathscr W\rightarrow E$, one has:
    \begin{align*}
      L(1_\Gamma,\kappa) u(y)&=LRTu(y)=-R\frac{du}{dx}(-y)\\
      &=-RTLu(y)=(1_\Gamma,\kappa)Lu(y),
    \end{align*}
   and  for $N_f$, 
    \begin{align*}
      N_f(\alpha,(1_\Gamma,\kappa)u)(y)&=f(\alpha,j(RTu)_y=f(\alpha,RTj(u)_{-y})\\
     \mbox{(by (P\ref{assu:P6}))}~~~~~\quad&=-Rf(\alpha,j(u)_{-y})=-RTN_f(\alpha,u)(y)\\
      &=(1_\Gamma,\kappa)N_f(\alpha,u)(y).
    \end{align*}
  \end{remark}
\vs
To convert \eqref{eq:3.2_op_reform} to a fixed-point problem in $\mathscr W,$ consider the operator $K:\mathscr W\rightarrow E$ defined by
\begin{align}
      K u :=\frac{1}{2\pi}\int_0^{2\pi}u(x)dx.
  \end{align}
Clearly, $L + K$ is invertible, but since $K$ is not $G$-equivariant  (and, therefore, $L+K$is not $G$-equivariant as well), one needs to ``adjust" the standard resolvent argument.
To this end, consider an additional linear operator $S:E\rightarrow E$ given by
 \begin{align}
 \label{eq:formula-S}
      Su(x):=\begin{bmatrix}
        0 & I \\
        I & 0
      \end{bmatrix}
      u(x)
  \end{align}
  \vs
  
  \begin{remark}\label{rem:SK} \rm It is easy to see that
    \begin{align}
      \label{eq:3.2_S_and_R}
      SR = -RS.
    \end{align}
 Hence,   
    \begin{align*}
      SK(1_\Gamma,\kappa)u(x)&=SKRTu(x)=SRKu(x)\\
      \mbox{(by \eqref{eq:3.2_S_and_R})}~~~\quad&=-RSKu(x)=-RTSKu(x)\\
      &=(1_\Gamma,\kappa)SKu(x),
    \end{align*}
form which it follows that $L+SK:\mathscr W\rightarrow E$ is $G$-equivariant. One can easily verify that $L+SK$ is also an isomorphism.
  \end{remark}
\vs

Combining Remark \ref{rem:SK} with \eqref{eq:3.2_op_reform}, one can reformulate \eqref{eq:fde-initial} 
as the folowing  $G$-equivariant fixed-point problem in $\mathscr W$:
  \begin{align}
    \label{eq:3.2_fa_reform1}
    u = \mathcal{F}(\alpha,u),
  \end{align}
where $   \mathcal{F}:\mathbb{R}\times \mathscr W\rightarrow \mathscr W$ is defined by
  \begin{align}
    \label{eq:def_mathcalF}
      \mathcal{F}(\alpha,u):=(L+SK)^{-1}\left[N_f(\alpha,u)+SKu\right].
  \end{align}
  \vs
  
Define $    \mathfrak{F}:\mathbb{R}\times \mathscr W\rightarrow \mathscr W$ by
  \begin{align}
    \label{eq:def_mathfrakF}
     \mathfrak{F}(\alpha,u) := u-\mathcal{F}(\alpha,u).
  \end{align}
Then, one can rewrite \eqref{eq:3.2_fa_reform1} in the equivalent form as
  \begin{align}
    \label{eq:3.2_fa_reform2}
    \mathfrak{F}(\alpha,u)=0.
  \end{align}
  \vs
  
\begin{remark}\label{rmk:3.2_Geq_of_F}\rm  Since $j$ is compact, 
$\mathfrak{F}$ is a $G$-equivariant compact vector field.
\end{remark}
\vs
 \section{$\Gamma\times O(2)$-Degree Method}
  \label{sec:Gdeg_method_for_bif}
  In this section, we will apply the equivariant degree method to detect
  and classify the branches of bifurcating $2\pi$-periodic solutions to system
  \eqref{eq:fde-initial} .
\vs

  \subsection{Linearization and Necessary Condition for the Bifurcation}
  Let us recall the following standard
\vs

  \begin{definition}\rm
    \label{def:bifurcation_branch}
Assume that the set
    \begin{align}
      \mathfrak{C}:=\overline{\{(\alpha,u)\in\Omega: \mathfrak{F}(\alpha,u)=0,\;u\neq 0\}}
    \end{align}
    contains a compact connected component $\mathscr{C}$ containing nontrivial $2\pi$-periodic functions and such that $(\alpha_o,0)\in\mathscr{C}$. Then, 
     $(\alpha_o,0)$ is  called a {\it bifurcation point} for \eqref{eq:fde-initial} and  $\mathscr C$ is said to be {\it branch} of nontrivial $2\pi$-periodic solutions
    to \eqref{eq:fde-initial} bifurcating from $(\alpha_o,0)$.
  \end{definition}
\vs 
  The lemma following below provides a {\it necessary} condition for $(\alpha_o,0)$ to be
  a bifurcation point.
  \vs 
  \begin{lemma}
    \label{lemma:necessary_condition}
    Under the assumptions (P\ref{assu:P1}), (P\ref{assu:P2}), (P\ref{assu:P3}) and (P\ref{assu:P5}),
    suppose that $(\alpha_o,0)$ is a bifurcation point for \eqref{eq:fde-initial}.
    Then, $(\alpha_o,0)\in\Psi$ (see \eqref{eq:Psi}).
  \end{lemma}
  \begin{proof}\label{proof:necessary_condition}
    Let
    \begin{align}
      \label{eq:4.1_linearization}
      a(\alpha):=D_u\mathfrak{F}(\alpha,0)=\id-(L+SK)^{-1}\left[D_uN_f(\alpha,0)+SK\right]
    \end{align}
    be the linearization of $\mathfrak{F}$ at $(\alpha,0)$. Then,
    \begin{align}
      a(\alpha)=\bigoplus_{i=0}^\infty{a_l(\alpha)},
    \end{align}
    where $a_l(\alpha):=a(\alpha)|_{\mathscr W_l}:\mathscr W_l\rightarrow \mathscr W_l$ is given by
    \begin{align}
      \label{eq:4.1_lin_restriction}
      a_l(\alpha):=a(\alpha)|_{\mathscr W_l}=\begin{cases}
        -S\left.D_uf(\alpha,0)\right\rvert_V,&l=0\\
        \id-\frac{1}{il}\left.D_uf(\alpha,0)\right\rvert_{\mathscr W_l},&l>0
      \end{cases}
    \end{align}
    (see \eqref{eq:fourier_mode}).
    By  Definition \ref{def:bifurcation_branch}, there exists a sequence $\{(\alpha_n,u_n)\}$ convergent to
    $(\alpha_o,0)$ such that $u_n\neq0$ and $\mathfrak F(\alpha_n,u_n)=0$ for all $n$.
    Hence, by conditions (P\ref{assu:P1}) and (P\ref{assu:P2}),
    \begin{align}
      u_n-D_u\mathcal F(\alpha_n,0)u_n+r(\alpha_n,u_n)=0,
    \end{align}
    where $r(\alpha_n,u_n)$ satisfies
    \begin{align}
      \label{eq:conv_r}
      \lim_{j\rightarrow\infty}\frac{r(\alpha_n,u_n)}{\left\lVert u_n\right\rVert}=0.
    \end{align}
    Then,
    \begin{align}
      \label{eq:fuj}
      \frac{u_n}{\left\lVert u_n\right\rVert}-D_u\mathcal{F}(\alpha_n,0)\frac{u_n}{\left\lVert u_n\right\rVert}
      +\frac{r(\alpha_n,u_n)}{\left\lVert u_n\right\rVert}=0.
    \end{align}
    Since $D_u\mathcal F(\alpha_o,0)$ is compact and $D_u\mathcal F(\alpha_n,0)$ depends
    continuously on the first component (cf.~(P\ref{assu:P2})), one has (up to choosing a subsequence)
    \begin{align*}
      &\lim_{n\rightarrow\infty}{D_u\mathcal F(\alpha_n,0)\frac{u_n}{\left\lVert u_n\right\rVert}}\notag\\
      =\,&\lim_{n\rightarrow\infty}\left(D_u\mathcal F(\alpha_n,0)\frac{u_n}{\left\lVert u_n\right\rVert}-
          D_u \mathcal F(\alpha_o,0)\frac{u_n}{\left\lVert u_n\right\rVert}+
          D_u\mathcal F(\alpha_o,0)\frac{u_n}{\left\lVert u_n\right\rVert}\right)\notag\\
      =\,&\lim_{n\rightarrow\infty}\left(D_u\mathcal F(\alpha_n,0)\frac{u_n}{\left\lVert u_n\right\rVert}-
          D_u\mathcal F(\alpha_o,0)\frac{u_n}{\left\lVert u_n\right\rVert}\right)+
          \lim_{n\rightarrow\infty}\left(D_u\mathcal F(\alpha_o,0)\frac{u_n}{\left\lVert u_n\right\rVert}\right)\notag\\
      =\,&\lim_{n\rightarrow\infty}\left(D_u\mathcal F(\alpha_o,0)\frac{u_n}{\left\lVert u_n\right\rVert}\right)=v_*,
    \end{align*}
    which by combining it with  \eqref{eq:conv_r} and  \eqref{eq:fuj}, implies
    \begin{align*}
      \lim_{n\rightarrow\infty}{\frac{u_n}{\left\lVert u_n\right\rVert}}=v_*\neq0.
    \end{align*}
Thus 
    \begin{align*}
      v_*-D_u\mathcal F(\alpha_o,0)v_*=0.
    \end{align*}
 But this implies that $\id-D_u\mathcal F(\alpha_o,0)$ is not invertible, which contradicts the condition  (P\ref{assu:P3}).
  \end{proof}
\vs

  \begin{remark}\rm
    Notice that, we did not assume that the map $f$ is continuously differentiable in a neighborhood
    of $0\in C_{2n}$ (see  (P\ref{assu:P2})), therefore one cannot apply the standard Implicit
    Function Theorem argument to prove Lemma \ref{lemma:necessary_condition}.
  \end{remark}
\vs

  \subsection{Sufficient Condition}
  To apply the equivariant degree method, we need:
  \begin{itemize}
  \item[(a)] to localize a potential bifurcation point $(\alpha_o,0)$ in a 
  $G$-invariant neighborhood $\Omega \subset \mathbb R \oplus \mathscr W$, 
  \item[(b)] to define a $G$-invariant auxiliary function $\zeta$ allowing us to detect nontrivial $2\pi$-periodic solutions to \eqref{eq:fde-initial} by applying ``augmented'' map $\mathfrak F_\zeta$.
  
  \item[(c)] to adjust the neighborhood $\Omega$ in order to make $\mathfrak F_\zeta$ an  
  $\Omega$-admissible $G$-equivariant map.
    \end{itemize}
    \vs 
    To begin, define 
  \begin{equation}\label{eq:Omega-r}
  \Omega(\rho,r):=\{(\alpha,u)\in \mathbb{R}\times \mathscr W :\left\lvert\alpha-\alpha_o\right\rvert<\rho,\,
        \left\lVert u\right\rVert<r\}.
  \end{equation}
\vs
Clearly  $  \Omega(\rho,r)$ is a $G$-invariant open and bounded neighborhood of $(\alpha_o,0)$. 

\vs

  \begin{lemma}\label{lemma:snbd}
    Under the assumptions (P\ref{assu:P1})--(P\ref{assu:P6}), 
    there exist $\rho$, $r>0$ such that the neighborhood $\Omega(\rho,r)$ given by \eqref{eq:Omega-r} satisfies the following conditions:
    \renewcommand{\labelenumi}{(\roman{enumi})}
    \begin{enumerate}
      \item $\overline{\Omega(\rho,r)}\cap\Psi=\{(\alpha_o,0)\}$;
      \item $\mathfrak{F}(\alpha,u)\neq 0$ for $(\alpha,u)\in\overline{\Omega(\rho,r)}$ with
        $\left\lvert\alpha-\alpha_0\right\rvert=\rho$ and $u\neq 0$.
    \end{enumerate}
  \end{lemma}
  \begin{proof}
 By (P\ref{assu:P4}), there exist $\rho > 0$ and $\tilde{r}>0$ such that 
 \begin{equation}\label{eq:Psi-intersect}
 \forall_{r>0}\;\;  r\le \tilde r\;\;\Rightarrow\;\; \overline{\Omega(\rho,r)}\cap\Psi = \{(\alpha_o,0)\},
\end{equation}
which implies  (i).  Assume, for contradiction, that for all 
 $r< \tilde{r}$, condition (ii) is not satisfied. Then,  there
 exists a sequence $\left\{(\alpha_o + \rho,u_n)\right\}\subset \Omega(\rho,r)$ such that $u_n\not=0$,  $\lim_{n\to\infty}\|u_n\|= 0$
and  $\mathfrak{F}(\alpha_o + \rho,u_n) = 0$. Next, applying the 
 the same argument as in the proof of Lemma \ref{lemma:necessary_condition}, 
 one can show that $u_n - D\mathcal F_u(\alpha_o + \rho, 0)$ is not invertible. Combining this with (P\ref{assu:P3}) yields
    $(\alpha_o + \rho, 0)\in \overline{\Omega(\rho,\tilde{r}}) \cap \Psi$ which contradicts \eqref{eq:Psi-intersect}.
  \end{proof}
\vs

 \begin{definition}\label{def:special-neighborhood}  \rm
 Denote by $\Omega:=\Omega(\rho,r)$ the neighborhood provided by Lemma \ref{lemma:snbd} and  call it a {\it special neighborhood} of $(\alpha_o,0)$.
 \end{definition}
\vs

Next, we need to introduce the auxiliary function $\zeta$. 
To this end, consider the following two subsets in $\overline{\Omega}$:
  \begin{align}
  \label{eq:sets-auxila}
    \partial_0&:=\left\{(\alpha,u)\in\overline{\Omega}:u=0\right\};\\
    \partial_r&:=\left\{(\alpha,u)\in\overline{\Omega}:\left\lVert u\right\rVert=r\right\}.
  \end{align}
Clearly, $\partial_0$ and $\partial_r$ are $G$-invariant
disjoint closed sets.
Therefore, there exists a $G$-invariant Urysohn function
  $\varsigma:\overline{\Omega}\rightarrow\mathbb{R}$ satisfying
  \begin{align}
    \label{eq:4.2_prop_of_auxfnt}
    \begin{cases}
      \varsigma(\lambda,u)>0&\mbox{for }(\lambda,u)\in\partial_r\\
      \varsigma(\lambda,u)<0&\mbox{for }(\lambda,u)\in\partial_0
    \end{cases}.
  \end{align}
We will call such a function $\varsigma$ an {\em auxiliary function} for $\mathfrak{F}$ on $\overline{\Omega}$.
Define the {\it augmented} $G$-equivariant map
$\mathfrak{F}_\varsigma:\overline{\Omega}\rightarrow \mathbb{R}\times \mathscr W$ by
  \begin{align}
    \label{eq:4.2_F_varsigma}
\mathfrak{F}_\varsigma(\alpha,u)=\big(\varsigma(\alpha,u),\mathfrak{F}(\alpha,u)\big).
  \end{align}
\vs

By Lemma \ref{lemma:snbd} and
\eqref{eq:4.2_prop_of_auxfnt}, the map  $\mathfrak{F}_\varsigma$ is $\Omega$-admissible. 
Hence, $\eqdeg{G}(\mathfrak{F}_\varsigma,\Omega)$ is well-defined (cf.~Remark \ref{deg-Leray-Schauder}).

\vs

  \begin{remark}\label{rmk:indep_aux}\rm
  $\eqdeg{G}(\mathfrak{F}_\varsigma,\Omega)$ is independent of the choice of $\varsigma$.
Indeed, if there are two auxiliary functions $\varsigma_1$ and $\varsigma_2$,
then  $h_t:=(1-t)\mathfrak{F}_{\varsigma_1}+t\mathfrak{F}_{\varsigma_2}$ is an $\Omega$-admissible homotopy
between $\mathfrak{F}_{\varsigma_1}$  and $\mathfrak{F}_{\varsigma_2}$, meaning that 
 $\eqdeg{G}(\mathfrak{F}_{\varsigma_1},\Omega)=\eqdeg{G}(\mathfrak{F}_{\varsigma_2},\Omega)$ (see the property  (G\ref{eqdegprop:homotopy})).

  \end{remark}
  We are now in a position to reduce the bifurcation problem for \eqref{eq:fde-initial}  to the computation
  of the degree $\eqdeg{G}(\mathfrak{F}_\varsigma,\Omega)$. More precisely, one has the following result.
  \begin{theorem}\label{thm:suff_cond}
Given system \eqref{eq:fde-initial}, suppose  that $f$ satisfies  P\ref{assu:P1})---(P\ref{assu:P6}).
Let $\mathfrak{F}$ be defined by \eqref{eq:def_mathcalF} and  \eqref{eq:def_mathfrakF}, and assume that 
the point $(\alpha_o,0)\in\Psi$ is provided by (P\ref{assu:P4}). In addition, let  
  $\Omega$  be a special neighborhood of $(\alpha_o,0)$
  (see Definition \ref{def:special-neighborhood}) and consider $\varsigma$ defined by 
  \eqref{eq:sets-auxila}--\eqref{eq:4.2_prop_of_auxfnt}.
  Then, the field $\mathfrak{F}_\varsigma$ defined by \eqref{eq:4.2_F_varsigma}
  is $G$-equivariant and $\Omega$-admissible, so the 
  equivariant degree 
 \begin{align}
      \eqdeg{G}(\mathfrak{F}_\varsigma,\Omega) = \sum_{(H)\in\Phi_0(G)}n_H(H)
 \end{align}
  is well-defined. Moreover, if  for some $(H_o)\in\Phi_0(G)$ the  coefficient 
  $n_{H_o}$ is non-zero (i.e., $n_{H_o}\not=0$), then there exists a branch
  $\mathscr C$ of nontrivial $2\pi$-periodic solutions
  to \eqref{eq:fde-initial} bifurcating from $(\alpha_o,0)$
  (cf.~Definition \ref{def:bifurcation_branch}) satisfying the conditions:
  \renewcommand{\labelenumi}{(\roman{enumi})}
  \begin{enumerate}
    \item $\mathscr{C}\cap\partial_r\neq\varnothing$;
    \item $\mathscr {C}\subset\mathbb{R}\times \mathscr W^{H_o}$ (i.e., periodic solutions belonging to $\mathscr{C}$ have symmetries at least 
    $(H_o)$);
    \item if, in addition, $(H_o)$ is a maximal orbit type in some $\mathscr W_l$, $l = 1,2,...$, then there are at least $|G/H_o|_{S^1}$ different branches of non-trivial periodic solutions with symmetries at least $(H_o)$ bifurcating from $(\alpha_o,0)$ (here $| \cdot |_{S^1}$ stands for the number of 
    $S^1$-orbits in $G/H_o$).   
    \end{enumerate}
  \end{theorem}
  \vs
  To prove Theorem \ref{thm:suff_cond}, we need the following statement.
  \vs
  \begin{proposition}[Kuratowski (see \cite{Kur})]\label{prop:kuratowski}
    Let $X$ be a metric space. Suppose $A,B\subset X$ are two disjoint closed sets and
    $K\subset X$ is compact such that $K$ intersects both $A$ and $B$. If $K$ doesn't
    contain a connected component $K_o$ which intersects both $A$ and $B$, then
    there exist two disjoint open sets $V_1$ and $V_2$ satisfying
    \renewcommand{\labelenumi}{\arabic{enumi}.} 
        \begin{enumerate}
    \item $A\subset V_1$ and $B\subset V_2$;
    \item $(A\cup B\cup K)\subset(V_1\cup V_2)$.
    \end{enumerate}
  \end{proposition}
  \begin{proof}(Proof of Theorem \ref{thm:suff_cond}.) Put
   \begin{align}
      K&:=\overline{\left\{(\alpha,u)\in\Omega:u\neq0,\,\mathfrak{F}(\alpha,u)=0\right\}},\\
      \varsigma_q(\alpha,u)&:=\left\lVert u\right\rVert-q,\quad 0\leq q\leq r.
    \end{align}
To take advantage of Proposition \ref{prop:kuratowski}, we need to show that $K$ intersects $\partial_0$ and $\partial_r$. 
Indeed, since $\varsigma_q$ is an auxiliary function for any $0<q<r$, it follows that 
    \begin{align}
      \eqdeg{G}(\mathfrak{F}_{\varsigma_q},\Omega)=\eqdeg{G}(\mathfrak{F}_\varsigma,\Omega)\neq0
    \end{align}
(see \ref{rmk:indep_aux}).
    By property (G\ref{eqdegprop:existence}), there exists  $(\alpha_q,u_q)\in\Omega$
    such that $\mathfrak{F}_{\varsigma_q}(\alpha_q,u_q)=0$; in particular,
    $\left\lVert u_q\right\rVert=q$. Observe that $\mathfrak{F}$ is a
    compact field, therefore, $\mathfrak{F}^{-1}(0)$ is compact, so is $K$.
    Now, by the standard compactness argument, there exist $(\alpha_0,0)\in(K\cap\partial_0)$
    and $(\alpha_r,u_r)\in(K\cap\partial_r)$.
   
  Assume now, by contradiction,  that there is  no compact connected set $K_o\subset K$ which intersects $\partial_0$
    and $\partial_r$ simultaneously. Then, according to Proposition \ref{prop:kuratowski}, there exist two disjoint open sets $N'$, 
   and  $N''$ such that 
    $\partial_0 \subset N'$, $\partial_r \subset C$ and $K \subset (N'\cup N'')$. Put
    \begin{align}\label{eq:K-prime}
      K'&:=\left\{u\in K:G(u)\cap N''=\varnothing\right\};\\
      K''&:=\left\{u\in K:G(u)\cap N'=\varnothing\right\}.
    \end{align}
By construction, $K'$ and $K''$ are invariant disjoint sets. Combining the openness of $N'$ and $N''$ with the continuity of the $G$-action, 
one can easily show  that $K'$ and $K''$ are closed. Hence, the sets
    \begin{align}
      Z'&:=\partial_0\cup K';\\
      Z''&:=\partial_r\cup K''
    \end{align}
 are also closed, invariant and disjoint.   
Therefore, there exists a $G$-invariant Urysohn function $\mu:\overline{\Omega}\rightarrow\mathbb{R}$
with
    \begin{align}
      \mu(\alpha,u)=\begin{cases}
        1,&\mbox{if }(\alpha,u)\in Z'\\
        0,&\mbox{if }(\alpha,u)\in Z''
      \end{cases}.
    \end{align}
    Take the auxiliary function
    \begin{align}\label{eq:function-mu}
      \begin{cases}
        \varsigma:\overline{\Omega}\rightarrow\mathbb{R}\\
        \varsigma(\alpha,u)=\left\lVert u\right\rVert-\mu(\alpha,u)r
      \end{cases}.
    \end{align}
    By the existence property, the set $\mathfrak{F}_\varsigma^{-1}(0)\subset K = K' \cup K''$ is non-empty. On the other hand, 
if $\mathfrak{F}_\varsigma(\alpha_*,u_*)=0$, then,  either $(\alpha_*,u_*)\in K'$ or
    $(\alpha_*,u_*)\in K''$ (recall, $K' \cap K'' = \emptyset$). If $(\alpha_*,u_*)\in K'$, then (by  \eqref{eq:K-prime}--\eqref{eq:function-mu})
    \begin{align*}
      \varsigma(\alpha_*,u_*)=0\Rightarrow\left\lVert u_*\right\rVert=r,
    \end{align*}
which contradicts the choice of $K'$, so $(\alpha,u)\not\in K'$.
    Similarly, $(\alpha,u)\notin K''$, and we arrive at the contradiction with
    $\mathfrak{F}_\varsigma^{-1}(0)=\varnothing$.
  \end{proof}
  
  \vs
  Under the assumptions of  Theorem \ref{thm:suff_cond},  introduce the following concept.

  \vs
  
  \begin{definition}\label{def:equiv_topo_bif_inv}\rm 
    Given a  point $(\alpha_o,0)\in\Psi$, put
    \begin{align}\label{eq:formula-bif-inv}
      \omega(\alpha_o):=\eqdeg{G}(\mathfrak{F}_\varsigma,\Omega)
    \end{align}
   and call it  the {\em local equivariant topological bifurcation invariant}
    at $(\alpha_o,0)$.
  \end{definition}
  \vs
  
  In the next subsection, we will give an effective formula for the computation of $\omega(\alpha_o)$.
  
  \vs
  
  \subsection{Computation of $\omega(\alpha_o)$}
  \label{subsec:computation}
  
  \medskip
\paragraph{\bf (a) Reduction to Product Formula.}  Using the standard argument based on Property (G5) of the equivariant degree  (cf.~\cite{KW}, section 8.5),
one can  easily establish the following formula:
  \begin{align}
    \label{eq:diff_deg_formula}
    \eqdeg{G}(\mathfrak{F}_\varsigma,\Omega)
        =\eqdeg{G}(\mathfrak{F}_-,B(0,r))-\eqdeg{G}(\mathfrak{F}_+,B(0,r)),
  \end{align}
where $B(0,r)\subset W$ is the ball centered at $0$ with radius $r$ and
  \begin{align}
    \begin{cases}
      \mathfrak{F}_\pm:\overline{B(0,r)}\rightarrow \mathscr W\\
      \mathfrak{F}_\pm(u):=\mathfrak{F}(\alpha_o\pm\rho,u)
    \end{cases}
  \end{align}
(provided that $r$ is  sufficiently small).  
Futhermore, by (P\ref{assu:P1}) and (P\ref{eqdegprop:homotopy}),
  \begin{align}
    \eqdeg{G}(\mathfrak{F}_\pm,B(0,r))=\eqdeg{G}(a_\pm,B(0,r)),
  \end{align}
  where $a_\pm:=a(\alpha_o\pm\rho)$ (provided that $r$ is  sufficiently small).  Next (cf.~\cite{AED}, section 9.2), we apply the finite-dimensional approximation to $a_\pm$:
  there exists $m\in\mathbb{N}$ such that $a_\pm$ is  $B(0,r)$-admissibly homotopic to
  \begin{align*}
    \widetilde{a}_\pm:=a_\pm|_{\mathscr W^m}+\id|_{(\mathscr W^m)^\perp},
  \end{align*}
where $W^m :=V\oplus\bigoplus_{l=1}^m{\mathscr W_l}\subset \mathscr W $ (see   \eqref{eq:fourier_mode}). 

  Therefore  (by (G\ref{eqdegprop:normalization}) and (G\ref{eqdegprop:multiplicity})),
 \begin{align}\label{eq:huy1}
    \eqdeg{G}(a_\pm,B(0,r))=\eqdeg{G}(a_\pm|_{\mathscr W^m},D^m),
  \end{align}
  where
 $D^m:=\{u\in \mathscr W^m\mid \left\lVert u\right\rVert<r\}$.
 Put $W_*^m:=\bigoplus_{l=1}^m{\mathscr W_l}$ and $D_*^m:=D^m\cap \mathscr W_*^m$, and let $D\subset V$ denote the unit ball.
  Clearly, $a_\pm|_{\mathscr W^m}=a_0^\pm\oplus a_*^\pm$, where $a_0^\pm :=a_0(\alpha_o\pm\rho)$ and $a_*^\pm:=\bigoplus_{l=1}^m{a_l(\alpha_o\pm\rho)}$.
 By (G\ref{eqdegprop:multiplicity}),
  \begin{align}
    \label{eq:4_deg_decomp}
    \eqdeg{G}(a_\pm |_{\mathscr W^m}, D^m)=\eqdeg{G}(a_0^\pm,D)\cdot
      \eqdeg{G}(a_*^\pm,D_*^m),
  \end{align}
  where the multiplication is taken in the Burnside ring $A(G)$.

\vs 
\paragraph{\bf (b) Computation of $\eqdeg{G}(a_0^\pm,D)$.}
  By \ref{assu:P3},
  \begin{align*}
    \eqdeg{G}(a_0^\pm,D)=\eqdeg{G}(a_0(\alpha_o),D).
  \end{align*}
Take the isotypical decompositions \eqref{eq:2.3_V_decomp} and \eqref{eq:2.3_Wkj_decomp}  and denote by  $\mathscr{W}_{j,0}^k$ the irreducible $G$-representation on which $\mathscr W_{j,0}^k$ is modeled.  
Then, using Proposition \ref{prop;basic-degree}, one obtains:
 \begin{align}
    \label{eq:4_degformula_0}
    \eqdeg{G}(a_0(\alpha_o),D) = 
    \prod_{\mu \in \sigma_-(a_0(\alpha_o))}\prod_{k=1}^2\prod_{j=1}^r(\deg_{\mathscr{W}_{j,0}^k})^{{m_{j,0}^k}(\mu)},
  \end{align}
where $\sigma_-$ is the negative spectrum and $m_{j,0}^k(\mu)$ stands for the $\mathscr{W}_{j,0}^k$-multiplicity of $\mu$.
\vs

\begin{remark}\label{rem:formula-prod1}\rm 
Observe (cf.~\eqref{eq:4.1_lin_restriction}) that $\mu \in \sigma_- (a_0(\alpha_o)) \iff \mu \in \sigma_+ (SD_uf(\alpha_o,0))$, where $S$ is given by  \eqref{eq:formula-S} and  $\sigma_+$ stands for the positive spectrum.
\end{remark}
\vs
\paragraph{\bf  (c) Computation of $\eqdeg{G}(a_*^\pm,D_*^m)$.} Using the 
  isotypical decompositions 
  \eqref{eq:2.3_V_decomp}, \eqref{eq:2.2_W_isodecomp} and \eqref{eq:Wjl-decomposition}
  (see also \eqref{eq:Wjlk}) and applying once again Proposition \ref{prop;basic-degree},
  one obtains:
  \begin{align}
    \label{eq:4_degformula-l}
    \eqdeg{G}(a_*^\pm,D_*^m) = 
    \prod_{\mu \in \sigma_-(a_*^\pm)}\prod_{l=1}^m\prod_{j=1}^r(\deg_{\mathscr{W}_{j,l}})^{{m_{j,l}^{\pm}}(\mu)},
  \end{align}
  where $m_{j,l}^\pm(\mu)$ stands for the $\mathscr{W}_{j,l}$-multiplicity
  of the eigenvalue $\mu$ of $a_*^\pm$.
  \vs
  \begin{remark}\label{rem:formula-prod2}\rm
    Put  $a_{jl}^\pm:=a_l(\alpha_o\pm\rho)|_{\mathscr W_{j,l}}$, $j = 1,...,r$ and $l = 1,...,m$.
    Then,  
    \begin{align*}
      \mu\in\sigma_-(a_*^\pm)&\iff \mu < 0 \; \;\text{and} \;\; \det(\mu\id-a_{jl}^\pm)=0 \\
      &\iff \mu < 0 \;\; \text{and}\; \det(\mu\id-a_l(\alpha_o\pm\rho)|_{\mathscr W_{j,l}})=0 \\
      &\iff \mu < 0 \;\; \text{and}\; \det\left((\mu-1)\id+\frac{1}{il}D_uf(\alpha_o\pm\rho,0)|_{\mathscr W_{j,l}}\right)=0\\
      &\iff \mu < 0 \;\; \text{and} \; \det\left(il(1-\mu)\id-D_uf(\alpha_o\pm\rho,0)|_{\mathscr W_{j,l}}\right)=0 
    \end{align*}
    for some $j = 1,...,r$ and  $l = 1,...,m$. Hence, $\mu\in\sigma_-(a_*^\pm)$ if and only if
    $il\xi$ is an eigenvalue of $D_uf(\alpha_o\pm\rho,0)|_{\mathscr W_{j,l}}$ for $\xi:=1-\mu>1$
    and some $j = 1,...,r$ and  $l = 1,...,m$. 
  \end{remark}
  \vs
  Combining \eqref{eq:formula-bif-inv}--\eqref{eq:4_degformula-l}, one arrives at the following formula:
  \begin{align}\label{coro:deg_formula}
    \omega(\alpha_o)=
      &\prod_{\mu \in \sigma_-(a_0(\alpha_o))}\prod_{k=1}^2\prod_{j=1}^r(\deg_{\mathscr{W}_{j,0}^k})^{{m_{j,0}^k}(\mu)}\\
      &\left(
        \prod_{\mu \in \sigma_-(a_*^-)}
        \prod_{l=1}^m
        \prod_{j=1}^r(\deg_{\mathscr{W}_{j,l}})^{{m_{j,l}^{-}}(\mu)} -
        \prod_{\mu \in \sigma_-(a_*^+)}\prod_{l=1}^m\prod_{j=1}^r(\deg_{\mathscr{W}_{j,l}})^{{m_{j,l}^{+}}(\mu)}
      \right).\notag
  \end{align}
  
\vs

\section{Computation of $\omega(\alpha_0)$: Examples}
  In this section, we will focus on the application of Theorem \ref{thm:suff_cond} to
  networks of oscillators of types \eqref{eq:dde_system} and \eqref{eq:ide_system} coupled symmetrically.
  To this end, it is enough to compute the equivariant topological bifurcation invariants
  (cf.~Definition \ref{def:equiv_topo_bif_inv} and Corolary \eqref{coro:deg_formula}) for both cases.
\vs 

 \subsection{Space-Reversal Symmetry for Second Order DDEs and IDEs}
  To begin with, let us show that \eqref{eq:dde_system} and \eqref{eq:ide_system} (provided
  that $k_\alpha$ is even) are reversible equations. Indeed,
  if $v$ is a solution to \eqref{eq:dde_system}, put $\mathfrak v (x)=v(-x)$. Then,
  \begin{align*}
   \ddot {\mathfrak v}(x)&=\ddot v(-x)\\
    &=g(v(-x))+a(v(-x-\alpha)+v(-x+\alpha))\\
    &=g(\mathfrak{v}(x))+a(\mathfrak{v}(x+\alpha)+\mathfrak{v}(x-\alpha)),
  \end{align*}
  therefore, $\mathfrak v$ is also a solution to \eqref{eq:dde_system}, thus,
  \eqref{eq:dde_system} is reversible.
  Similarly, assume $v$ is a solution to \eqref{eq:ide_system} and
  $k_\alpha$ is even. Then,
  \begin{align*}
   \ddot{ \mathfrak{v}}(x)&=v''(-x)\\
    &=g(v(-x))+a\int_{-\infty}^\infty{v(-x-y)k_\alpha(y)dy}\\
    &=g(\tilde{v}(x))+ a\int_{-\infty}^\infty{v(-x+y)k_\alpha(-y)dy}\\
    &=g(\mathfrak{v}(x))+a\int_{-\infty}^\infty{\mathfrak{v}(x-y)k_\alpha(y)dy},
  \end{align*}
  therefore, $\mathfrak v$ is also a solution.
 \vs
  In what follows, we will assume for simplicity that $g$ is a scalar function  and 
  \begin{align}
    \label{eq:def_g}
    g(u)=c_1u+q(u),
  \end{align}
  where $c_1<0$ is a parameter characterizing, for example, heat sink (see, for example, \cite{Arnol'd}),
  while $q$ is a continuous function differentiable at $0$ with $q'(0)=0$. Also, put $a=1$.
\vs

  \subsection{Coupling Systems with Octahedral Symmetry and First Order Reformulation}
  \label{subsec:coupling_matrix}
  As a case study, in what follows, we consider networks of eight identical oscillators
  of type \eqref{eq:dde_system} (resp.~\eqref{eq:ide_system}) coupled in the cube-like
  configuration (cf.~\cite{BKRZ2012}). In this way, any permutation of vertices,
  which preserves the coupling between them, is a symmetry of this configuration.
  To simplify our exposition, we consider the symmetry group of cube consisting
  of all transformations preserving orientation in $\mathbb R^3$.
  Obviously, this group is isomorphic to $S_4$.
  Therefore, in the context relevant to Subsections \ref{subsec:hypo} and \ref{subsec:g_rep}, we have to deal with
  the 8-dimensional permutational $\Gamma$-representation $V^1$ with
  $\Gamma=S_4$.

\vs

  As it is well-known (see, for example, \cite{AED}, \cite{FultonHarris}), $V_1$ admits
  the $S_4$-isotypical decomposition $V^1=V_1\oplus V_2\oplus V_3\oplus V_4$,
  where $V_1$ is the trivial 1-dimensional representation,
  $V_2$ is the 1-dimensional $S_4$-representation where $S_4$ acts as $S_4/A_4$,
  $V_3$ is the natural 3-dimensional irreducible representation of $S_4$ as a subgroup of $SO(3)$,
  and $V_4=V_2\otimes V_3$.

\vs
  Combining the conservation laws with the coupling symmetry allows us
  to choose the coupling matrix in the form $c_2B$, where $B$ is
  given by
  \begin{align}
    \label{eq:def_B}
    B=\begin{bmatrix}
      -3&1&0&1&1&0&0&0\\
      1&-3&1&0&0&1&0&0\\
      0&1&-3&1&0&0&1&0\\
      1&0&1&-3&0&0&0&1\\
      1&0&0&0&-3&1&0&1\\
      0&1&0&0&1&-3&1&0\\
      0&0&1&0&0&1&-3&1\\
      0&0&0&1&1&0&1&-3
    \end{bmatrix}
  \end{align}
  (the parameter $c_2>0$ characterizes the strength of coupling; see \cite{BKRZ2012} for similar
  considerations, where the network of 8 coupled van der Pol oscillators was studied).
\vs

  In order to be compatible with Sections \ref{sec:assu_and_reform} and \ref{sec:Gdeg_method_for_bif},
  we need to reformulate \eqref{eq:dde_system} and \eqref{eq:ide_system}
  as first order differential equations. To this end, put
  \begin{align}
    \label{eq:u1u2}
    u(x):=\begin{bmatrix}
      u_1(x)\\
      u_2(x)
    \end{bmatrix}=
    \begin{bmatrix}
      v(x)\\
      v'(x)
    \end{bmatrix}.
  \end{align}
  Then, coupling oscillators of type \eqref{eq:dde_system} using \eqref{eq:def_B} yields the following
  system:
  \begin{align}
    \label{eq:dde_reform}
    \frac{d}{dx}u(x)=
    \begin{bmatrix}
      u_2(x)\\
      c_1u_1(x)+q(u_1(x))+c_2Bu_1(x)+u_1(x+\alpha)+u_1(x-\alpha)
    \end{bmatrix}.
  \end{align}
\vs

  Similarly, combining \eqref{eq:ide_system} with \eqref{eq:def_B} yields:
  \begin{align}
    \label{eq:ide_reform}
    \frac{d}{dx}u(x)=
    \begin{bmatrix}
      u_2(x)\\
      c_1u_1(x)+q(u_1(x))+c_2Bu_1(x)+\int_{-\infty}^\infty u_1(x-y)k_\alpha(y)\, dy
    \end{bmatrix}.
  \end{align}

\vs
  \subsection{Equivariant Spectral Information}
  First, consider the equation \eqref{eq:dde_reform}.
  Following the scheme described in Subsection \ref{subsec:reform}, reformulate equation \eqref{eq:dde_reform}
  as the operator equation \eqref{eq:3.2_fa_reform2} (see also \eqref{eq:3.2_fa_reform1},
  \eqref{eq:def_mathcalF} and \eqref{eq:def_mathfrakF}).
  In order to compute the topological equivariant bifurcation invariant by formula \eqref{coro:deg_formula},
  we need the following information: 
  \begin{itemize}
  \item[(i)]  eigenvalues of $a$ (cf.~Lemma \ref{lemma:necessary_condition})
  together with their $\mathscr W_{j,l}$-multiplicities, 
  \item[(ii)] basic degrees $\deg_{\mathscr W_{j,l}}$, and
  
  \item[(iii)] multiplication formulae for the Burnside ring $A(S_4\times O(2))$. 
  \end{itemize}
  \vs
  Below, we focus on (i),
  while (ii) and (iii) are explained in Appendices \ref{appendix:a} and \ref{appendix:b}.

\vs

  By direct computation (see  \eqref{eq:4.1_lin_restriction}), one has
  \begin{align}
    a_l(\alpha)=\begin{cases}
      -\begin{bmatrix}
        D_xg(0)+2\id&0\\
        0&\id
      \end{bmatrix}, & l=0\\
      \begin{bmatrix}
        \id & -\dfrac{1}{il}\id\\
        -\dfrac{1}{il}(D_xg(0)+2\cos(\alpha l)\id) & \id
      \end{bmatrix}, & l>0
    \end{cases}.
  \end{align}
 Thus, for the characteristic equation, we obtain
  \begin{align}
    \det(a_0(\alpha)-\lambda\id)&=\det\begin{bmatrix}
      -D_xg(0)-(2+\lambda)\id & 0 \\
      0 & -(1+\lambda)\id
    \end{bmatrix}
  \end{align}
  and
  \begin{align}
    \det(a_l(\alpha)-\lambda\id)&=\det\begin{bmatrix}
      (1-\lambda)\id & -\dfrac{1}{il}\id \\
      -\dfrac{1}{il}(D_xg(0)+2\cos(\alpha l)\id) & (1-\lambda)\id
    \end{bmatrix}\notag\\
    &=\det\left[(1-\lambda)^2\id+\dfrac{1}{l^2}(D_xg(0)+2\cos(\alpha l)\id)\right]\notag\\
    &=\det\left[((1-\lambda)^2+\dfrac{1}{l^2}2\cos(\alpha l))\id+\dfrac{1}{l^2}D_xg(0)\right].
  \end{align}
\vs

  \begin{table}[h]
    \centering
    \begin{tabular}{cp{4cm}lc}
      \toprule
      & isotypical component $X$ (modeled on $\mathcal X$) & Eigenvalue of $a|_X$ & $\mathcal X$-multiplicity \\
      \midrule
      \multirow{8}{*}{$l=0$} & $\mathscr W_{1,0}^1$ & $\mu_{1,0}$ & 1 \\
      & $\mathscr W_{2,0}^1$ & $\mu_{2,0}$ & 1 \\
      & $\mathscr W_{3,0}^1$ & $\mu_{3,0}$ & 1 \\
      & $\mathscr W_{4,0}^1$ & $\mu_{4,0}$ & 1 \\
      & $\mathscr W_{1,0}^2$ & $-1$ & 1 \\
      & $\mathscr W_{2,0}^2$ & $-1$ & 1 \\
      & $\mathscr W_{3,0}^2$ & $-1$ & 1 \\
      & $\mathscr W_{4,0}^2$ & $-1$ & 1 \\
      \midrule
      \multirow{8}{*}{$l>0$} & $\mathscr W_{1,l}$ & $\mu_{1,l}^+$ & 1 \\
      & & $\mu_{1,l}^-$ & 1 \\
      & $\mathscr W_{2,l}$ & $\mu_{2,l}^+$ & 1 \\
      & & $\mu_{2,l}^-$ & 1 \\
      & $\mathscr W_{3,l}$ & $\mu_{3,l}^+$ & 1 \\
      & & $\mu_{3,l}^-$ & 1 \\
      & $\mathscr W_{4,l}$ & $\mu_{4,l}^+$ & 1 \\
      & & $\mu_{4,l}^-$ & 1 \\
      \bottomrule
    \end{tabular}
    \caption{Eigenvalues of $a(\alpha)$ (MDDE)}
    \label{table:eigenvalue_dde}
  \end{table}
    \begin{table}[h]
    \centering
    \begin{tabular}{cp{4cm}lc}
      \toprule
      & isotypical component $X$ (modeled on $\mathcal X$) & Eigenvalue of $a|_X$ & $\mathcal X$-multiplicity \\
      \midrule
      \multirow{8}{*}{$l=0$} & $\mathscr W_{1,0}^1$ & $\mu_{1,0}$ & 1 \\
      & $\mathscr W_{2,0}^1$ & $\mu_{2,0}$ & 1 \\
      & $\mathscr W_{3,0}^1$ & $\mu_{3,0}$ & 1 \\
      & $\mathscr W_{4,0}^1$ & $\mu_{4,0}$ & 1 \\
      & $\mathscr W_{1,0}^2$ & $-1$ & 1 \\
      & $\mathscr W_{2,0}^2$ & $-1$ & 1 \\
      & $\mathscr W_{3,0}^2$ & $-1$ & 1 \\
      & $\mathscr W_{4,0}^2$ & $-1$ & 1 \\
      \midrule
      \multirow{4}{*}{$l>0$} & $\mathscr W_{1,l}$ & $\mu_{1,l}^+$ & 1 \\
      & & $\mu_{1,l}^-$ & 1 \\
      & $\mathscr W_{2,l}$ & $\mu_{2,l}^+$ & 1 \\
      & & $\mu_{2,l}^-$ & 1 \\
      & $\mathscr W_{3,l}$ & $\mu_{3,l}^+$ & 1 \\
      & & $\mu_{3,l}^-$ & 1 \\
      & $\mathscr W_{4,l}$ & $\mu_{4,l}^+$ & 1 \\
      & & $\mu_{4,l}^-$ & 1 \\
      \bottomrule
    \end{tabular}
    \caption{Eigenvalues of $a(\alpha)$ (IDE)}
    \label{table:eigenvalue_ide}
  \end{table}\ \\

  As the result, $a$ has the eigenvalues presented in Table \ref{table:eigenvalue_dde},
  where
  \begin{align}
    \mu_{j,0}=-c_1+2(j-1)c_2-2
  \end{align}
  and
  \begin{align}
    \label{eq:eigenvalue_l}
    \mu_{j,l}^{\pm}=
      1\pm\dfrac{1}{l}\sqrt{-c_1+2(j-1)c_2-2\cos(l\alpha)}
  \end{align}
  for $l>0$ and $j=1,2,3,4$. Since the bifurcation may take place
  only if $\mu_{j,l}^{\pm}=0$, it follows that only $\mu_{j,l}^-$ may
  give rise to the bifurcation points. In addition, suppose $\mu_{j,l}^-(\alpha_o)=0$.
  If $\frac{d}{d\alpha}\cos(\alpha_o)=0$, then $\alpha_o$ is not a bifurcation point. On the other hand,
  if $\frac{d}{d\alpha}\cos(\alpha_o)>0$, then $\mu_{j,l}^-(\alpha_o-\delta)>0$ and $\mu_{j,l}^-(\alpha_o+\delta)<0$
  for any sufficiently small $\delta>0$; if $\cos'(\alpha_o)<0$, then
  $\mu_{j,l}^-(\alpha_o-\delta)<0$ and $\mu_{j,l}^-(\alpha_o+\delta)>0$ for any
  sufficiently small $\delta>0$. Hence, if $\frac{d}{d\alpha}\cos(\alpha_o)\neq0$, then $\alpha_o$ is a bifurcation point and
  one can effectively evaluate the negative spectra required by formula \eqref{coro:deg_formula}.
\vs

  Consider now  \eqref{eq:ide_reform}  and assume, for the sake of definiteness, that 
  \begin{equation}\label{eq:sink-heat} 
  k_\alpha(x)=\dfrac{1}{\sqrt{2\pi\alpha^2}}e^{-\frac{x^2}{2\alpha^2}}
  \end{equation}  
  (heat kernel). Then, similarly to the MDDE case, one can compute the  eigenvalues of $a(\alpha)$, which are presented in Table \ref{table:eigenvalue_ide},
  where
  \begin{align}
    \mu_{j,0}=-c_1+2(j-1)c_2-1
  \end{align}
  and
  \begin{align}
    \label{eq:eigenvalue_l}
    \mu_{j,l}^{\pm}=
      1\pm\dfrac{1}{l}\sqrt{-c_1+2(j-1)c_2-e^{-(l\alpha)^2/2}}
  \end{align}
  for $l>0$ and $j=1,2,3,4$. In this case (in contrast to the DDE case),
  if $\mu_{j,l}^-(\alpha_o)=0$, then $\mu_{j,l}^-(\alpha_o-\delta)<0$
  and $\mu_{j,l}^-(\alpha_o+\delta)>0$ (recall that $\alpha_o>0$).
  Hence, $\alpha_o$ is always a bifurcation point.
\vs
  \subsection{Parameter Space and Bifurcation Mechanism}\label{subsec:parameter_space}
  In this subsection, we present graphical illustrations of  the considered in the above examples bifurcations.
  Since, in both examples, $a=a(c_1,c_2,\alpha)$ depends on three parameters,
  take $\mathbb R^3$ with coordinates $(c_1,c_2,\alpha)$ and consider a bounded connected set
  $P\subset\mathbb R^3$.
  
\vs
  By assumption (P\ref{assu:P3}), $a_0$ is non-singular, i.e., $\mu_{j,0}\neq 0$ for any $j$. The family of  ``surfaces'' $\{T_j\}$, defined by the equation $\mu_{j,0}= 0$,  has to be excluded from $P$.

\vs

 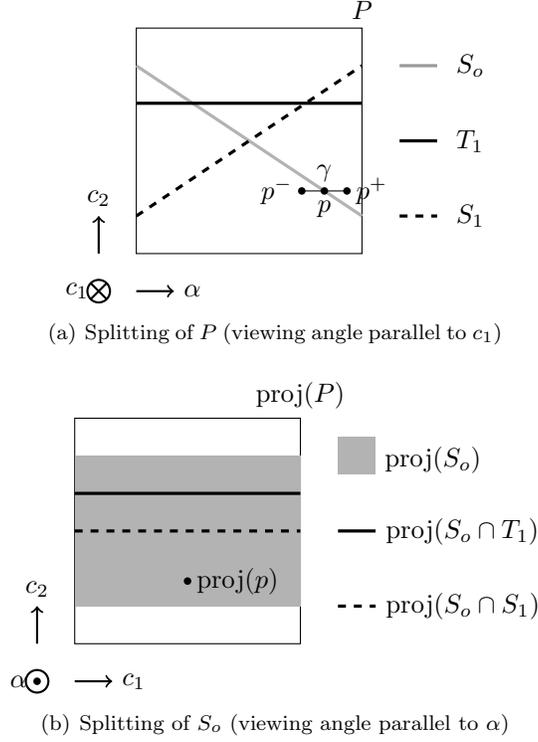
\begin{figure}[h]
    \centering
    \subfigure[Splitting of $P$ (viewing angle parallel to $c_1$)]{
      \begin{tikzpicture}
        \draw[thick, ->] (.5,0)--(1,0) node[right,pos=1] {$\alpha$};
        \draw[thick, ->] (0,.5)--(0,1) node[above,pos=1] {$c_2$};
        \draw[thick] (0,0) circle (.15) node[left] {$c_1$};
        \draw[thick] (135:.15)--(315:.15);
        \draw[thick] (45:.15)--(225:.15);

        \draw (.5,.5) rectangle +(3,3) node[above] {$P$};

        \draw[very thick, draw=black!30] (3.5,1)--(.5,3);

        \draw[very thick] (.5,2.5)--(3.5,2.5);

        \draw[very thick, dashed] (.5,1)--(3.5,3);

        \fill (3,4/3) circle (.05) node[below] {$p$};
        \fill (2.7,4/3) circle (.05) node[left] {$p^-$};
        \fill (3.3,4/3) circle (.05) node[right] {$p^+$};

        \draw (2.7,4/3)--node[above]{$\gamma$}(3.3,4/3);

        \draw[very thick, draw=black!40] (4,3) -- (4.5,3) node[right] {~~$S_o$}; 
        \draw[very thick]  (4,2) -- (4.5,2) node[right] {~~$T_1$};
        \draw[very thick, dashed] (4,1) -- (4.5,1) node[right] {~~$S_1$};
      \end{tikzpicture}
      \label{subfig:split_P}
    }
    \hspace{.1em}
    \subfigure[Splitting of $S_o$ (viewing angle parallel to $\alpha$)]{
      \begin{tikzpicture}
        \draw[thick, ->] (.5,0)--(1,0) node[right,pos=1] {$c_1$};
        \draw[thick, ->] (0,.5)--(0,1) node[above,pos=1] {$c_2$};
        \draw[thick] (0,0) circle (.15) node[left] {$\alpha$};
        \fill (0,0) circle (.05);

        \draw (.5,.5) rectangle +(3,3) node[above] {$\mathrm{proj}(P)$};

        \fill[fill=black!30] (.5,1) rectangle +(3,2);

        \draw[very thick] (.5,2.5)--(3.5,2.5);

        \draw[very thick, dashed] (.5,2)--(3.5,2);

        \fill (2,4/3) circle (.05) node[right] {$\mathrm{proj}(p)$};

        \fill[fill=black!30] (4,2.75) rectangle +(.5,.5) node[anchor=north west] {$\mathrm{proj}(S_o)$};
        \draw[very thick] (4,2) -- (4.5,2) node[right] {$\mathrm{proj}(S_o\cap T_1)$};
        \draw[very thick, dashed] (4,1) -- (4.5,1) node[right] {$\mathrm{proj}(S_o\cap S_1)$};
      \end{tikzpicture}
      \label{subfig:split_S}
    }
    \caption{Illustration of $P$ and the Bifurcation Surfaces}
    \label{fig:illu_P}
  \end{figure}

  On the other hand, the necessary condition for the occurrence of the
  bifurcation is that
  $\mu_{j,l}^-=0$ for some $j$ and $l>0$. Thus, the equations $\mu_{j,l}^-=0$  
  determines a family of ``surfaces'' $\{S_{j,l}\}$ in $P$, where potential bifurcation points are located.
  In other words, $\bigcup S_{j,l}$ represents
  the set of such potential bifurcation points in $P$, which we will call {\it critical set}.
In addition, 
  the intersection set of two or more different surfaces $\{S_{j,l}\}$ 
 will be called a {\it collision set}.
\vs

  To illustrate the bifurcation mechanism,  assume that $P\subset\mathbb R^3$ is a cube with the removed  $\{T_j\}:=\{T_1\}$, where the critical set $\{S_{j,l}\}$ is given as $\{S_o,S_1\}$ (see  Figure \ref{subfig:split_P}).
  Note that the critical set splits $P$ into open connected regions
  in which $a(\alpha,c_1,c_2)$ is non-singular. By the homotopy property of the equivariant
  degree, $\eqdeg{G}(a(\alpha,c_1,c_2),D)$ admits a single value in each of those  regions.

\vs
  Take a path $\gamma(t)=(c_1(t),c_2(t),\alpha(t))$ crossing $S_o$
  at $p$. Then, in a sufficiently small neighborhood of $p$,
  $\gamma$ joins two points $p^-$ and $p^+$ belonging
  to different regions. Hence, the bifurcation invariant is
  $\omega(p)=\eqdeg{G}(a(p^-))-\eqdeg{G}(a(p^+))\neq0$ and the bifurcation takes place.
  Observe that in the case study, $c_1(t)\equiv c_1$, $c_2(t)\equiv c_2$ and $\alpha(t)=t$.

\vs
  Unfortunately, in the above three-dimensional set $P$, it may be difficult
  to visualize how the critical set splits $P$.  Keeping in mind that the bifurcation may occur only on
one of  $S_{j,l}$, 
  we can project one of those surfaces  $S_o$ in $\{S_{j,l}\}$ onto  $\mathrm{proj}(S_o)$ in $c_1$-$c_2$ plane and study the sets 
  $\mathrm{proj}(S_o\cap T_1)$   (the projection of the steady-state)  and $\mathrm{proj}(S_o\cap S_1)$ (the projection of the collision set).
  We refer to Figure \ref{subfig:split_S}, where three grey area shows the projection of $S_o$ and
  the dashed line represents the projection of the collision set. Observe, that for any point $(c_1,c_2)$ in the regions (I)--(VII), by changing the value of the parameter $\alpha$ one crosses the critical set, therefore for those critical points we are getting  different values of
  the bifurcation invariant. Also, the remaining part of the
  square (i.e., outside the grey region) corresponds to $(c_1, c_2)$ for which there
  is no bifurcation.
  In the next subsection, we will apply this procedure to obtain the exact values of the local equivariant bifurcation invariants for the 
  examples related to MDDE and IDE systems.
\vs

  \subsection{Results}
  In what follows, we use the notations introduced in subsection  \ref{subsec:parameter_space}.
  \subsubsection*{(a) Mixed Delay Differential Equation}
  Take $P=\{(c_1,c_2,\alpha):-2.5<c_1<0,\;0<c_2<1.5,\;0<\alpha<\pi\}$ and $S_o=S_{2,1}$.
  Then, the projection of $S_o$ to the $c_1$-$c_2$ plane is as follows (cf.~Figure \ref{subfig:split_S}).
  \begin{figure}[h]
    \centering
    \includegraphics[scale=0.5]{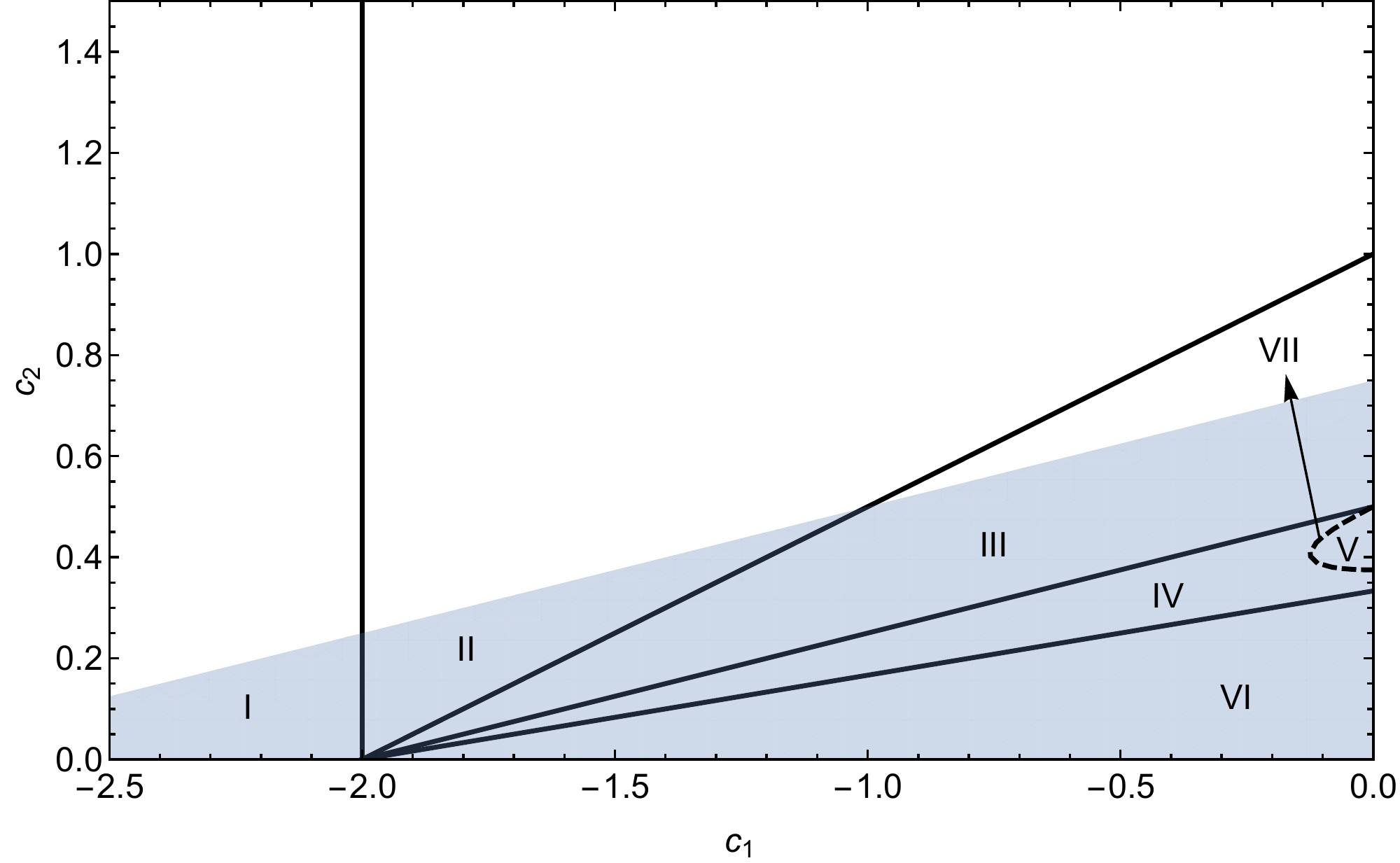}
    \caption{Splitting of $\mathrm{proj}(S_o)$ associated to the MDDE system}
    \label{fig:split_S_DDE}
  \end{figure}\ \\
  \begin{table}[h]
    \centering
    \begin{tabular}{cl}
      \toprule
        Region & Condition \\
      \midrule
        I & $(-c_1+4c_2<3)\wedge(-c_1>2)$ \\
        II & $(-c_1+4c_2<3)\wedge(-c_1<2)\wedge(-c_1+2c_2>2)$ \\
        III & $(-c_1+4c_2<3)\wedge(-c_1+2c_2<2)\wedge(-c_1+4c_2>2)$ \\
        IV & $(-c_1+4c_2<2)\wedge(-c_1+6c_2>2)\wedge(-c_1+6c_2-(-c_1+4c_2-1)^2<2)$ \\
        V & $(-c_1+6c_2-(-c_1+4c_2-1)^2>2)$ \\
        VI & $(-c_1+6c_2<2)$ \\
        VII & $(-c_1+6c_2-(-c_1+4c_2-1)^2=2)$ \\
      \bottomrule
    \end{tabular}
  \end{table}\ \\

  Given $c_1$ and $c_2$ in the shaded area, the (unique) critical value of the bifurcation parameter associated to $S_o$
  is equal to $\alpha_o(c_1,c_2)=\arccos((-c_1+4c_2-1)/2)$.
\vs

  The values of the equivariant bifurcation invariant $\omega(\alpha_o)$ are summarized in the following table.
{\floatsetup[table]{font=scriptsize}
  \begin{table}[h]
    \centering
   \scalebox{.87}{  \begin{tabular}{cp{13cm}}
      \toprule
        Region & $\omega(\alpha_o)$ \\
      \midrule
        I & $ (\amall{D_4}{D_{2}}{\mathbb{Z}_{2}}{\mathbb{Z}_4}{})
             -(\amall{D_3}{D_{2}}{\mathbb{Z}_{2}}{\mathbb{Z}_3}{})
             +(\amall{D_2}{D_{2}}{\mathbb{Z}_{2}}{D_1}{})
             -(\amall{D_2}{D_{2}}{\mathbb{Z}_{2}}{\mathbb{Z}_2}{})
             +(\amall{D_1}{D_{2}}{\mathbb{Z}_{2}}{\mathbb{Z}_1}{})
             -(\amall{\mathbb{Z}_2}{D_{2}}{\mathbb{Z}_{2}}{\mathbb{Z}_1}{})
             -(\amall{D_4}{D_{4}}{D_{4}}{\mathbb{Z}_1}{})
             +(\amall{D_3}{D_{3}}{D_{3}}{\mathbb{Z}_1}{})
             -(\amall{D_4}{D_{2}}{D_{2}}{\mathbb{Z}_2}{V_4})
             -(\amall{D_4}{D_{1}}{D_{1}}{\mathbb{Z}_4}{})
             +(\amall{D_2}{D_{1}}{D_{1}}{\mathbb{Z}_2}{})
             -(\amall{D_1}{D_{1}}{D_{1}}{\mathbb{Z}_1}{})
             +(\amall{\mathbb{Z}_2}{D_{1}}{D_{1}}{\mathbb{Z}_1}{})$ \\
        II & $-(\amall{D_4}{D_{2}}{\mathbb{Z}_{2}}{\mathbb{Z}_4}{})
              +(\amall{D_3}{D_{2}}{\mathbb{Z}_{2}}{\mathbb{Z}_3}{})
              -(\amall{D_2}{D_{2}}{\mathbb{Z}_{2}}{D_1}{})
              +(\amall{D_2}{D_{2}}{\mathbb{Z}_{2}}{\mathbb{Z}_2}{})
              -(\amall{D_1}{D_{2}}{\mathbb{Z}_{2}}{\mathbb{Z}_1}{})
              +(\amall{\mathbb{Z}_2}{D_{2}}{\mathbb{Z}_{2}}{\mathbb{Z}_1}{})
              +(\amall{D_4}{D_{4}}{D_{4}}{\mathbb{Z}_1}{})
              -(\amall{D_3}{D_{3}}{D_{3}}{\mathbb{Z}_1}{})
              +(\amall{D_4}{D_{2}}{D_{2}}{\mathbb{Z}_2}{V_4})
              +(\amall{D_4}{D_{1}}{D_{1}}{\mathbb{Z}_4}{})
              -(\amall{D_2}{D_{1}}{D_{1}}{\mathbb{Z}_2}{})
              +(\amall{D_1}{D_{1}}{D_{1}}{\mathbb{Z}_1}{})
              -(\amall{\mathbb{Z}_2}{D_{1}}{D_{1}}{\mathbb{Z}_1}{})$ \\
        III & $-(\amall{D_4}{D_{2}}{\mathbb{Z}_{2}}{\mathbb{Z}_4}{})
               -(\amall{D_3}{D_{2}}{\mathbb{Z}_{2}}{\mathbb{Z}_3}{})
               +(\amall{D_2}{D_{2}}{\mathbb{Z}_{2}}{D_1}{})
               +(\amall{D_1}{D_{2}}{\mathbb{Z}_{2}}{\mathbb{Z}_1}{})
               -(\amall{D_1}{D_{1}}{\mathbb{Z}_{1}}{D_1}{})
               +(\amall{D_4}{D_{4}}{D_{4}}{\mathbb{Z}_1}{})
               +(\amall{D_3}{D_{3}}{D_{3}}{\mathbb{Z}_1}{})
               +(\amall{D_4}{D_{2}}{D_{2}}{\mathbb{Z}_2}{V_4})
               -(\amall{D_2}{D_{2}}{D_{2}}{\mathbb{Z}_1}{\mathbb{Z}_2})
               +(\amall{D_4}{D_{1}}{D_{1}}{\mathbb{Z}_4}{})
               -(\amall{D_2}{D_{1}}{D_{1}}{\mathbb{Z}_2}{})
               -(\amall{\mathbb{Z}_2}{D_{1}}{D_{1}}{\mathbb{Z}_1}{})$ \\
        IV & $-(\amall{D_4}{D_{2}}{\mathbb{Z}_{2}}{\mathbb{Z}_4}{})
              -(\amall{D_3}{D_{2}}{\mathbb{Z}_{2}}{\mathbb{Z}_3}{})
              +(\amall{D_2}{D_{2}}{\mathbb{Z}_{2}}{D_1}{})
              +(\amall{D_1}{D_{2}}{\mathbb{Z}_{2}}{\mathbb{Z}_1}{})
              +(\amall{\mathbb{Z}_4}{D_{1}}{\mathbb{Z}_{1}}{\mathbb{Z}_4}{})
              +(\amall{\mathbb{Z}_3}{D_{1}}{\mathbb{Z}_{1}}{\mathbb{Z}_3}{})
              -(\amall{\mathbb{Z}_1}{D_{1}}{\mathbb{Z}_{1}}{\mathbb{Z}_1}{})
              +(\amall{D_4}{D_{4}}{D_{4}}{\mathbb{Z}_1}{})
              +(\amall{D_3}{D_{3}}{D_{3}}{\mathbb{Z}_1}{})
              +(\amall{D_4}{D_{2}}{D_{2}}{\mathbb{Z}_2}{V_4})
              -(\amall{D_2}{D_{2}}{D_{2}}{\mathbb{Z}_1}{\mathbb{Z}_2})
              +(\amall{D_4}{D_{1}}{D_{1}}{\mathbb{Z}_4}{})
              -(\amall{\mathbb{Z}_4}{D_{1}}{D_{1}}{\mathbb{Z}_2}{})
              -(\amall{D_2}{D_{1}}{D_{1}}{\mathbb{Z}_2}{})
              -(\amall{D_1}{D_{1}}{D_{1}}{\mathbb{Z}_1}{})
              -(\amall{\mathbb{Z}_2}{D_{1}}{D_{1}}{\mathbb{Z}_1}{})$ \\
        V & $-(\amall{D_4}{D_{2}}{\mathbb{Z}_{2}}{\mathbb{Z}_4}{})
             -(\amall{D_3}{D_{2}}{\mathbb{Z}_{2}}{\mathbb{Z}_3}{})
             +(\amall{D_2}{D_{2}}{\mathbb{Z}_{2}}{D_1}{})
             +(\amall{V_4}{D_{2}}{\mathbb{Z}_{2}}{\mathbb{Z}_2}{})
             +(\amall{D_1}{D_{2}}{\mathbb{Z}_{2}}{\mathbb{Z}_1}{})
             -(\amall{\mathbb{Z}_2}{D_{2}}{\mathbb{Z}_{2}}{\mathbb{Z}_1}{})
             +(\amall{\mathbb{Z}_4}{D_{1}}{\mathbb{Z}_{1}}{\mathbb{Z}_4}{})
             -(\amall{\mathbb{Z}_2}{D_{1}}{\mathbb{Z}_{1}}{\mathbb{Z}_2}{})
             -(\amall{D_4}{D_{4}}{D_{4}}{\mathbb{Z}_1}{})
             +(\amall{D_3}{D_{3}}{D_{3}}{\mathbb{Z}_1}{})
             +(\amall{D_4}{D_{1}}{D_{1}}{\mathbb{Z}_4}{})
             +(\amall{D_3}{D_{1}}{D_{1}}{\mathbb{Z}_3}{})
             -(\amall{V_4}{D_{1}}{D_{1}}{\mathbb{Z}_2}{})
             -2(\amall{D_1}{D_{1}}{D_{1}}{\mathbb{Z}_1}{})
             +(\amall{\mathbb{Z}_2}{D_{1}}{D_{1}}{\mathbb{Z}_1}{})$ \\
        VI & $-(\amall{D_4}{D_{2}}{\mathbb{Z}_{2}}{\mathbb{Z}_4}{})
              -(\amall{D_3}{D_{2}}{\mathbb{Z}_{2}}{\mathbb{Z}_3}{})
              +(\amall{D_2}{D_{2}}{\mathbb{Z}_{2}}{D_1}{})
              +(\amall{V_4}{D_{2}}{\mathbb{Z}_{2}}{\mathbb{Z}_2}{})
              +(\amall{D_1}{D_{2}}{\mathbb{Z}_{2}}{\mathbb{Z}_1}{})
              -(\amall{\mathbb{Z}_2}{D_{2}}{\mathbb{Z}_{2}}{\mathbb{Z}_1}{})
              +(\amall{\mathbb{Z}_4}{D_{1}}{\mathbb{Z}_{1}}{\mathbb{Z}_4}{})
              -(\amall{\mathbb{Z}_2}{D_{1}}{\mathbb{Z}_{1}}{\mathbb{Z}_2}{})
              +(\amall{D_4}{D_{4}}{D_{4}}{\mathbb{Z}_1}{})
              +(\amall{D_3}{D_{3}}{D_{3}}{\mathbb{Z}_1}{})
              +(\amall{D_4}{D_{2}}{D_{2}}{\mathbb{Z}_2}{V_4})
              -(\amall{D_2}{D_{2}}{D_{2}}{\mathbb{Z}_1}{\mathbb{Z}_2})
              -(\amall{V_4}{D_{2}}{D_{2}}{\mathbb{Z}_1}{})
              +(\amall{D_4}{D_{1}}{D_{1}}{\mathbb{Z}_4}{})
              -(\amall{\mathbb{Z}_4}{D_{1}}{D_{1}}{\mathbb{Z}_2}{})
              -(\amall{D_2}{D_{1}}{D_{1}}{\mathbb{Z}_2}{})
              -(\amall{V_4}{D_{1}}{D_{1}}{\mathbb{Z}_2}{})
              -(\amall{D_1}{D_{1}}{D_{1}}{\mathbb{Z}_1}{})
              +(\amall{\mathbb{Z}_2}{D_{1}}{D_{1}}{\mathbb{Z}_1}{})$ \\
        VII & $-(\amall{V_4}{D_{2}}{\mathbb{Z}_{2}}{\mathbb{Z}_2}{})
               +(\amall{\mathbb{Z}_2}{D_{2}}{\mathbb{Z}_{2}}{\mathbb{Z}_1}{})
               +(\amall{\mathbb{Z}_3}{D_{1}}{\mathbb{Z}_{1}}{\mathbb{Z}_3}{})
               +(\amall{\mathbb{Z}_2}{D_{1}}{\mathbb{Z}_{1}}{\mathbb{Z}_2}{})
               -(\amall{\mathbb{Z}_1}{D_{1}}{\mathbb{Z}_{1}}{\mathbb{Z}_1}{})
               +2(\amall{D_4}{D_{4}}{D_{4}}{\mathbb{Z}_1}{})
               +(\amall{D_4}{D_{2}}{D_{2}}{\mathbb{Z}_2}{V_4})
               -(\amall{D_2}{D_{2}}{D_{2}}{\mathbb{Z}_1}{\mathbb{Z}_2})
               -(\amall{D_3}{D_{1}}{D_{1}}{\mathbb{Z}_3}{})
               -(\amall{\mathbb{Z}_4}{D_{1}}{D_{1}}{\mathbb{Z}_2}{})
               -(\amall{D_2}{D_{1}}{D_{1}}{\mathbb{Z}_2}{})
               +(\amall{V_4}{D_{1}}{D_{1}}{\mathbb{Z}_2}{})
               +(\amall{D_1}{D_{1}}{D_{1}}{\mathbb{Z}_1}{})
               -2(\amall{\mathbb{Z}_2}{D_{1}}{D_{1}}{\mathbb{Z}_1}{})$ \\
      \bottomrule
    \end{tabular}}
    \caption{Values of $\omega(\alpha_o)$ in different regions}
    \label{table:bifinv_dde}
  \end{table}}\\

  Let us explain how the information provided by \ref{table:bifinv_dde} can be used to classify symmetric properties
  of bifurcating branches of $2\pi$-periodic solutions to system \eqref{eq:dde_reform} as well as to estimate the
  minimal number of these branches. To simplify our exposition, we restrict ourselves with the case
  of region II only.\\
\vs 
  \begin{theorem}\label{thm:dde-num}
    Let $(c_1,c_2)$ be a point in region II. Then:
    \renewcommand{\labelenumi}{(\roman{enumi})}
    \begin{enumerate}
    \item $(\alpha_o,0)$, where $\alpha_o=\arccos((-c_1+4c_2)/2)$,
      is a bifurcation point of $2\pi$-periodic solutions to system \eqref{eq:dde_reform};
    \item there exist at least:
    \begin{itemize} 
    \item 3 bifurcating branches with symmetry at least
      $(\amall{D_4}{D_4}{D_4}{\mathbb Z_1}{})$, 
      \item 6 bifurcating branches with symmetry
      at least $(\amall{D_2}{D_2}{\mathbb Z_2}{D_1}{})$ and 
      \item 3 bifurcating
      branches with symmetry at least $(\amall{D_4}{D_2}{\mathbb Z_2}{\mathbb Z_4}{})$
      
   \end{itemize} 
   (cf.~Subsection \ref{subsec:direct_product_subgroup} and  Appendix \ref{subsec:ccss_S4xO2}).\end{enumerate}
  \end{theorem}

  \begin{proof}
    Observe that all the orbit types appearing in $\omega(\alpha_o)$ in the
    considered case are related to the first Fourier mode. Among them,
    the following orbit types are maximal in $\mathscr W_1$: 
      $(\amall{D_4}{D_4}{D_4}{\mathbb Z_1}{})$,
      $(\amall{D_2}{D_2}{\mathbb Z_2}{D_1}{})$ and
      $(\amall{D_4}{D_2}{\mathbb Z_2}{\mathbb Z_4}{})$ (cf.~Table \ref{table:ccsstbl_S4xO2}).
    To complete the proof, it remains to observe that if $(\amal{H}{K}{L}{\phi}{\psi})$
    is a (maximal) orbit type in $\mathscr W_1$, then there are $\left\lvert S_4/H\right\rvert$ different
    $S^1$-orbits of $2\pi$-periodic solutions to system \eqref{eq:dde_reform} (provided
    $\amal{H}{K}{L}{\phi}{\psi}$ appears in $\omega(\alpha_o)$).
  \end{proof}
\vs
  \subsubsection*{(b) Integro-Differential Equation}

  Take $P=\{(c_1,c_2,\alpha):-1.25<c_1<0,\;0<c_2<1.25,\;0<\alpha<\infty\}$ and $S_o=S_{3,1}$.
  Then, the projection of $S_o$ to the $c_1$-$c_2$ plane is as follows (cf.~\ref{subfig:split_S}):
  \begin{figure}[h]
    \centering
    \includegraphics[scale=0.5]{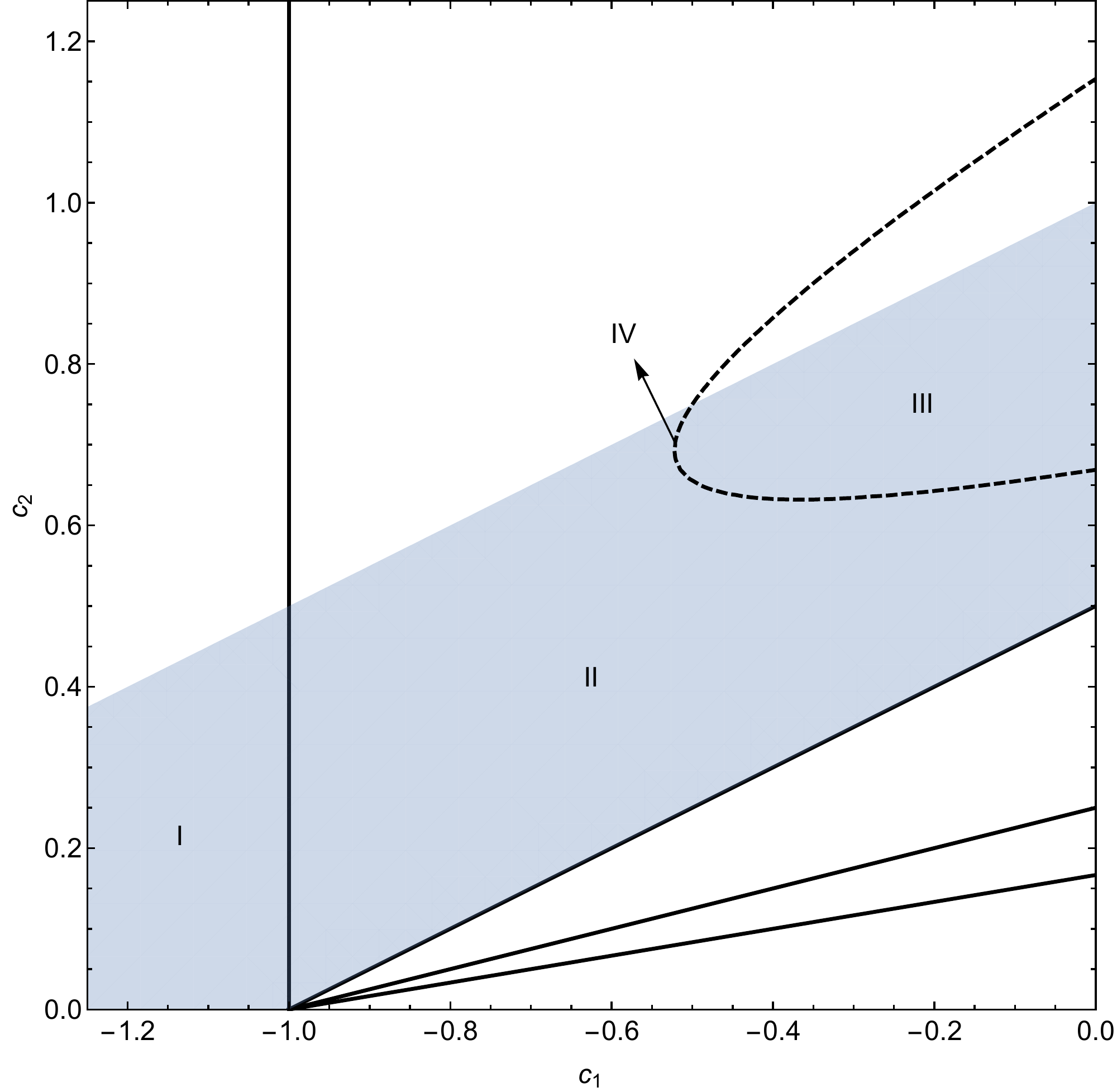}
    \caption{Splitting of $\mathrm{proj}(S_o)$ associated to the IDE system}
    \label{fig:split_S_IDE}
  \end{figure}\ \\
  Formally, this projection is given by 
  \begin{align}
    \mathrm{proj}(S_o)=\{(c_1,c_2):-1.25<c_1<0,\,0<c_2<1.25\,\mbox{and }-c_1+2c_2<2\},
  \end{align}
  which splits into the regions described by the table below.

\vs

  \begin{table}[h]
    \centering
    \begin{tabular}{cl}
      \toprule
        Region & Condition \\
      \midrule
        I & $-c_1>1$ \\
        II & $-c_1+6c_2-4>(-c_1+2c_2-1)^4$ \\
        III & $-c_1+6c_2-4<(-c_1+2c_2-1)^4$ \\
        IV & $-c_1+6c_2-4=(-c_1+2c_2-1)^4$ \\
      \bottomrule
    \end{tabular}
    \caption{Description of each region in $\mathrm{proj}(S_o)$}
  \end{table}\ \\

  Given $(c_1,c_2)\in\mathrm{proj}(S_o)$ (shaded area), the (unique)
  cirtical value of the bifuraction parameter associated to $S_o$
  is equal to  $\alpha_o(c_1,c_2)=\sqrt{-2\ln(-c_1+2c_2-1)}$.
  The values of the equivariant bifurcation invariant $\omega(\alpha_o)$ are summarized in the following table.
 
  {\floatsetup[table]{font=scriptsize}
  \begin{table}[h]
    \centering
   \scalebox{.87}{  \begin{tabular}{cp{13cm}}
      \toprule
        Region & $\omega(\alpha_o)$\\
      \midrule
        I & $ (\amall{D_4}{D_{2}}{\mathbb{Z}_{2}}{D_2}{})
             -(\amall{\mathbb{Z}_4}{D_{2}}{\mathbb{Z}_{2}}{\mathbb{Z}_2}{})
             -(\amall{D_2}{D_{2}}{\mathbb{Z}_{2}}{D_1}{})
             -(\amall{V_4}{D_{2}}{\mathbb{Z}_{2}}{\mathbb{Z}_2}{})
             +(\amall{\mathbb{Z}_2}{D_{2}}{\mathbb{Z}_{2}}{\mathbb{Z}_1}{})
             +(\amall{D_3}{D_{1}}{\mathbb{Z}_{1}}{D_3}{})
             -(\amall{D_1}{D_{1}}{\mathbb{Z}_{1}}{D_1}{})
             +(\amall{\mathbb{Z}_2}{D_{1}}{\mathbb{Z}_{1}}{\mathbb{Z}_2}{})
             +(\amall{D_4}{D_{4}}{D_{4}}{\mathbb{Z}_1}{})
             -(\amall{D_3}{D_{3}}{D_{3}}{\mathbb{Z}_1}{})
             +(\amall{D_2}{D_{2}}{D_{2}}{\mathbb{Z}_1}{D_1})
             +(\amall{D_2}{D_{2}}{D_{2}}{\mathbb{Z}_1}{\mathbb{Z}_2})
             -(\amall{D_4}{D_{1}}{D_{1}}{D_2}{})
             +(\amall{\mathbb{Z}_4}{D_{1}}{D_{1}}{\mathbb{Z}_2}{})
             +(\amall{V_4}{D_{1}}{D_{1}}{\mathbb{Z}_2}{})
             -2(\amall{\mathbb{Z}_2}{D_{1}}{D_{1}}{\mathbb{Z}_1}{})$ \\
        II & $-(\amall{D_4}{D_{2}}{\mathbb{Z}_{2}}{D_2}{})
              +(\amall{\mathbb{Z}_4}{D_{2}}{\mathbb{Z}_{2}}{\mathbb{Z}_2}{})
              +(\amall{D_2}{D_{2}}{\mathbb{Z}_{2}}{D_1}{})
              +(\amall{V_4}{D_{2}}{\mathbb{Z}_{2}}{\mathbb{Z}_2}{})
              -(\amall{\mathbb{Z}_2}{D_{2}}{\mathbb{Z}_{2}}{\mathbb{Z}_1}{})
              -(\amall{D_3}{D_{1}}{\mathbb{Z}_{1}}{D_3}{})
              +(\amall{D_1}{D_{1}}{\mathbb{Z}_{1}}{D_1}{})
              -(\amall{\mathbb{Z}_2}{D_{1}}{\mathbb{Z}_{1}}{\mathbb{Z}_2}{})
              -(\amall{D_4}{D_{4}}{D_{4}}{\mathbb{Z}_1}{})
              +(\amall{D_3}{D_{3}}{D_{3}}{\mathbb{Z}_1}{})
              -(\amall{D_2}{D_{2}}{D_{2}}{\mathbb{Z}_1}{D_1})
              -(\amall{D_2}{D_{2}}{D_{2}}{\mathbb{Z}_1}{\mathbb{Z}_2})
              +(\amall{D_4}{D_{1}}{D_{1}}{D_2}{})
              -(\amall{\mathbb{Z}_4}{D_{1}}{D_{1}}{\mathbb{Z}_2}{})
              -(\amall{V_4}{D_{1}}{D_{1}}{\mathbb{Z}_2}{})
              +2(\amall{\mathbb{Z}_2}{D_{1}}{D_{1}}{\mathbb{Z}_1}{})$ \\
        III & $-(\amall{D_4}{D_{2}}{\mathbb{Z}_{2}}{D_2}{})
               +(\amall{\mathbb{Z}_4}{D_{2}}{\mathbb{Z}_{2}}{\mathbb{Z}_2}{})
               +(\amall{D_2}{D_{2}}{\mathbb{Z}_{2}}{D_1}{})
               -(\amall{D_3}{D_{1}}{\mathbb{Z}_{1}}{D_3}{})
               +(\amall{\mathbb{Z}_3}{D_{1}}{\mathbb{Z}_{1}}{\mathbb{Z}_3}{})
               +(\amall{D_1}{D_{1}}{\mathbb{Z}_{1}}{D_1}{})
               -(\amall{\mathbb{Z}_1}{D_{1}}{\mathbb{Z}_{1}}{\mathbb{Z}_1}{})
               +(\amall{D_4}{D_{4}}{D_{4}}{\mathbb{Z}_1}{})
               +(\amall{D_3}{D_{3}}{D_{3}}{\mathbb{Z}_1}{})
               +(\amall{D_4}{D_{2}}{D_{2}}{\mathbb{Z}_2}{V_4})
               -(\amall{D_2}{D_{2}}{D_{2}}{\mathbb{Z}_1}{D_1})
               +(\amall{D_4}{D_{1}}{D_{1}}{D_2}{})
               +(\amall{D_3}{D_{1}}{D_{1}}{\mathbb{Z}_3}{})
               -(\amall{\mathbb{Z}_4}{D_{1}}{D_{1}}{\mathbb{Z}_2}{})
               -2(\amall{D_1}{D_{1}}{D_{1}}{\mathbb{Z}_1}{})
               -(\amall{\mathbb{Z}_2}{D_{1}}{D_{1}}{\mathbb{Z}_1}{})$ \\
        IV & $ (\amall{V_4}{D_{2}}{\mathbb{Z}_{2}}{\mathbb{Z}_2}{})
              -(\amall{\mathbb{Z}_2}{D_{2}}{\mathbb{Z}_{2}}{\mathbb{Z}_1}{})
              -(\amall{\mathbb{Z}_3}{D_{1}}{\mathbb{Z}_{1}}{\mathbb{Z}_3}{})
              -(\amall{\mathbb{Z}_2}{D_{1}}{\mathbb{Z}_{1}}{\mathbb{Z}_2}{})
              +(\amall{\mathbb{Z}_1}{D_{1}}{\mathbb{Z}_{1}}{\mathbb{Z}_1}{})
              -2(\amall{D_4}{D_{4}}{D_{4}}{\mathbb{Z}_1}{})
              -(\amall{D_4}{D_{2}}{D_{2}}{\mathbb{Z}_2}{V_4})
              -(\amall{D_2}{D_{2}}{D_{2}}{\mathbb{Z}_1}{\mathbb{Z}_2})
              -(\amall{D_3}{D_{1}}{D_{1}}{\mathbb{Z}_3}{})
              -(\amall{V_4}{D_{1}}{D_{1}}{\mathbb{Z}_2}{})
              +2(\amall{D_1}{D_{1}}{D_{1}}{\mathbb{Z}_1}{})
              +3(\amall{\mathbb{Z}_2}{D_{1}}{D_{1}}{\mathbb{Z}_1}{})$ \\
      \bottomrule
    \end{tabular}}
  \end{table}}\ \\

  Similarly to  the MDDE case, we can conclude this subsection by the theorem.
\vs
  \begin{theorem}\label{thm:ide-num}
    Let $(c_1,c_2)$ be a point in region II. Then:
    \renewcommand{\labelenumi}{(\roman{enumi})}
    \begin{enumerate}
    \item $(\alpha_o,0)$, where $\alpha_o=\sqrt{-2\ln(-c_1+2c_2-1)}$,
      is a bifurcation point of $2\pi$-periodic solutions to system \eqref{eq:ide_reform};
    \item there exist at least:
    \begin{itemize}
    \item 3 bifurcating branches with symmetry 
      at least $(\amall{D_4}{D_4}{D_4}{\mathbb Z_1}{})$, 
      \item 6 bifurcating branches with symmetry
      at least $(\amall{D_2}{D_2}{\mathbb Z_2}{D_1}{})$ and 
      \item 3 bifurcating
      branches with symmetry at least $(\amall{D_4}{D_2}{\mathbb Z_2}{\mathbb Z_4}{})$
      \end{itemize}
      (cf.~subsection \ref{subsec:direct_product_subgroup} and  Appendix \ref{subsec:ccss_S4xO2}).
    \end{enumerate}
  \end{theorem}


  \appendix
 \section{Identification of $\Phi_0(\Gamma\times O(2))$}
  \label{appendix:a}
  \subsection{Conjugacy Classes of Subgroups in $\Gamma\times O(2)$}

  Consider the following diagram which is equivalent to Figure \ref{fig:subg_com_diagram1}:
  \begin{figure}[h]
    \centering
    \begin{tikzpicture}
      \node (S) at (0,-1) {$S$};
      \node (H) at (-2,-3) {$H$};
      \node (K) at (2,-3) {$K$};
      \node (QK) at (0,-3) {$K/Z_K$};
      \draw[->, thick] (S)--node[left] {$\tilde{\pi}_\mathcal{H}$} (H);
      \draw[->, thick] (S)--node[right] {$\tilde{\pi}_\mathcal{K}$} (K);
      \draw[->, thick] (H)--node[below] {$\varphi$} (QK);
      \draw[->, thick] (K)--(QK);
    \end{tikzpicture}
    \caption{Descibing the subgroup $S\le\Gamma\times O(2)$. $Z_K$ normal subgroup of $K$.}
    \label{fig:subg_com_diagram3}
  \end{figure}
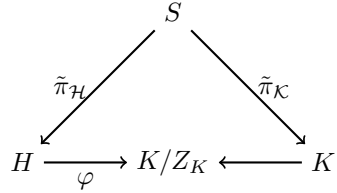\ \\
  In addition to the notations introduced in Subsections  \ref{subsec:jargon} and 
  \ref{subsec:direct_product_subgroup}, we denote
  by $S_1\sim_G S_2$ if $S_1$ is conjugate to $S_2$ in $G$.
  \vs
  Given two groups $\mathcal H$ and
  $\mathcal K$,  assume  
 $\Phi(\mathcal H)$ and $\Phi(\mathcal K)$ are known. 
 Let us show  how to identify $\Phi(\mathcal G)$
  for $\mathcal G:=\mathcal H\times\mathcal K$.
\vs

  Suppose $S_1$, $S_2<\mathcal G$ are given by
  \begin{align*}
    S_1&=\amal{H_1}{K_1}{L}{\varphi_1}{\psi_1},\\
    S_2&=\amal{H_2}{K_2}{L}{\varphi_2}{\psi_2}.
  \end{align*}
  Then $S_1\sim_\mathcal G S_2$ implies $H_1\sim_\mathcal H H_2$ and
  $K_1\sim_\mathcal K K_2$. Therefore, to identify $\Phi(\mathcal{G})$,
  it suffices to use  the following algorithm.
  \begin{algorithm}[h]
    \caption{Identify $\Phi(\mathcal G)$}
    \begin{algorithmic}
      \REQUIRE $\Phi(\mathcal H)$, $\Phi(\mathcal K)$;
      \STATE{$\Phi(\mathcal G)\gets\{\}$};
      \FOR{$(H)\in\Phi(\mathcal H)$}
        \FOR{$(K)\in\Phi(\mathcal K)$}
          \STATE{$\mathcal A=\mathcal A(H,K)\gets$ $\{\mbox{subgroups of $\mathcal G$ induced by $H$ and $K$}\}$};
          \STATE{Identify CCSs of $\mathcal G$ in $\mathcal A$};
          \STATE{$\Phi(\mathcal G)\gets\Phi(\mathcal G)\cup\{\mbox{CCSs in }\mathcal A\}$};
        \ENDFOR
      \ENDFOR
    \end{algorithmic}
    \label{alg:identify_PHI_G}
  \end{algorithm}\ \\
  In the procedure above, one needs to take only one representative from each CCS in
  $\Phi(\mathcal H)$ and $\Phi(\mathcal K)$.
  The only remaining question is to identify CCSs of $\mathcal G$ in $\mathcal A$.
  Keeping in mind  Figures \ref{fig:subg_com_diagram1} and \ref{fig:subg_com_diagram3},
  the answer to this question is provided by the following proposition.
\vs

  \begin{proposition} Let $\mathcal{H}$ and $\mathcal{K}$ be two groups and 
    $\mathcal G:=\mathcal{H}\times\mathcal{K}$. Then,  two subgroups 
    $S_1=\amal{H}{K}{L}{\varphi_1}{\psi_1}$ and
    $S_2=\amal{H}{K}{L}{\varphi_2}{\psi_2}$ of $\mathcal{G}$ 
    are conjugate if and only if there exists $(a,b)\in\mathcal G$ such that
    the diagrams shown in Figure \ref{fig:conj_dig} commute.
    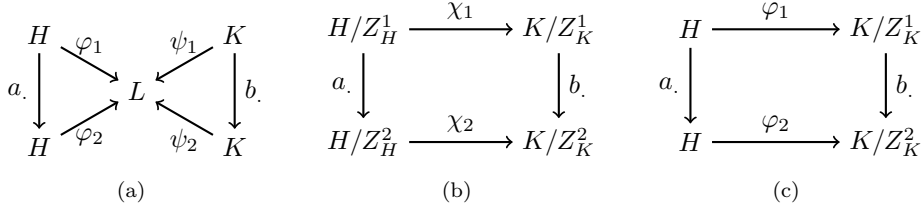
\begin{figure}[h]
      \centering
      \subfigure[]{
        \begin{tikzpicture}
          \node (H1) at (150:1.5) {$H$};
          \node (H2) at (210:1.5) {$H$};
          \node (K1) at (30:1.5) {$K$};
          \node (K2) at (330:1.5) {$K$};
          \node (L) at (0:0) {$L$};
          \draw[->, thick] (H1)--node[left]{$a_.$}(H2);
          \draw[->, thick] (K1)--node[right]{$b_.$}(K2);
          \draw[->, thick] (H1)--node[above]{$\varphi_1$}(L);
          \draw[->, thick] (H2)--node[below]{$\varphi_2$}(L);
          \draw[->, thick] (K1)--node[above]{$\psi_1$}(L);
          \draw[->, thick] (K2)--node[below]{$\psi_2$}(L);
        \end{tikzpicture}
      }\hfill
      \subfigure[]{
        \begin{tikzpicture}
          \node (QH1) at (150:1.5) {$H/Z_{H}^1$};
          \node (QH2) at (210:1.5) {$H/Z_{H}^2$};
          \node (QK1) at (30:1.5) {$K/Z_{K}^1$};
          \node (QK2) at (330:1.5) {$K/Z_{K}^2$};
          \draw[->, thick] (QH1)--(QK1) node[pos=0.5, above] {$\chi_1$};
          \draw[->, thick] (QH2)--(QK2) node[pos=0.5, above] {$\chi_2$};
          \draw[->, thick] (QH1)--(QH2) node[pos=0.5, left] {$a_.$};
          \draw[->, thick] (QK1)--(QK2) node[pos=0.5, right] {$b_.$};
        \end{tikzpicture}
      }\hfill
      \subfigure[]{
        \begin{tikzpicture}
          \node (H1) at (150:1.5) {$H$};
          \node (H2) at (210:1.5) {$H$};
          \node (QK1) at (30:1.5) {$K/Z_{K}^1$};
          \node (QK2) at (330:1.5) {$K/Z_{K}^2$};
          \draw[->, thick] (H1)--(QK1) node[pos=0.5, above] {$\varphi_1$};
          \draw[->, thick] (H2)--(QK2) node[pos=0.5, above] {$\varphi_2$};
          \draw[->, thick] (H1)--(H2) node[pos=0.5, left] {$a_.$};
          \draw[->, thick] (QK1)--(QK2) node[pos=0.5, right] {$b_.$};
        \end{tikzpicture}
      }
      \caption{Diagrams of Conjugate Subgroups in $\mathcal G=\mathcal H\times\mathcal K$}
      \label{fig:conj_dig}
    \end{figure}\ \\
    In Figure \ref{fig:conj_dig}(b), $\chi_i$ is an isomorphism induced by
    $\varphi_i$ and $\psi_i$, $i=1,2$.
    In addition, if such a pair $(a,b)$ exists, then $a\in N_\mathcal H(H)$
    and $b\in N_\mathcal K(K)$.
  \end{proposition}
\vs
  In our applications, we focus on $G=\Gamma\times O(2)$, where $\Gamma$ is a finite group.
  Being inspired by Figure \ref{fig:conj_dig}(c), we suggest a modified procedure for this case (cf.
  Algorithm 1).

  \begin{algorithm}[h]
    \caption{Identify $\Phi(G)$}
    \begin{algorithmic}
      \REQUIRE $\Phi(\Gamma)$, $\Phi(O(2))$;
      \STATE{$\Phi(G)\gets\{\}$};
      \FOR{$(H)\in\Phi(\Gamma)$}
        \FOR{$K/Z_K$ (up to isomorphism) where $Z_K\triangleleft K$ and $(K) \in\Phi(O(2))$}
          \STATE{$\mathcal A=\mathcal A(H,K,Z_K)\gets\{\amal{H}{K}{L}{\phi}{\psi}<G:\ker(\psi) = Z_K, \, L \simeq K/Z_K \}$};
          \STATE{$\mathcal I=\mathcal I(H,K,Z_K)\gets\{\mbox{epimorhisms from $H$ to $K/Z_K$}\}$};
          \STATE{identify conjugacy classes of epimorhpisms (CCEs) in $\mathcal I$;
              (CCSs in $\mathcal A$ are identified accordingly)}
          \STATE{$\Phi(\mathcal G)\gets\Phi(\mathcal G)\cup\{\mbox{CCSs in }\mathcal A\}$};
        \ENDFOR
      \ENDFOR
    \end{algorithmic}
    \label{alg:identify_PHI_G_alt}
  \end{algorithm}\ \\
\vs
In the procedure above, we say that two epimorphisms
  $\varphi_1,\varphi_2:H\rightarrow K/Z_K$ are conjugate if  
  they satisfy the property described in Figure \ref{fig:conj_dig}(c). Following this definition,
  we identify $\mathcal A$ with $\mathcal I$, i.e., 
  there is a one-to-one correspondence between CCSs in $\mathcal A$
  and CCEs in $\mathcal I$. This allows us  to deal with a {\it finite}  set of isomorphism classes 
  of quotient groups $K/Z_K$, $Z_K \triangleleft K < O(2)$, rather than dealing with the {\it infinite} set
  $\Phi(O(2))$ (cf.~Table \ref{table:quotient_groups}).
  Next, we can test whether $H$ can be epimorphic to a given $K/Z_K$ by GAP.\footnote{
  GAP is a system for computational discrete algebra.}
  Since $H$ can only be epimorphic to a smaller group, it suffices to consider finitely many $K/Z_K$.
  Finally, given $H$ and $Z_K \triangleleft K < O(2)$, GAP can be used to identify conjugacy classes of
  epimorphisms from $H$ to $K/Z_K$.

  
  %
  \begin{table}[h]
    \label{table:quotient_groups}
    \centering
    \begin{tabular}{ccc}
      \toprule
      $K/Z_K$ & $K$ & $Z_K$ \\
      \midrule
      \multirow{4}{*}{$\mathbb Z_1$} & $O(2)$ & $O(2)$ \\
      & $SO(2)$ & $SO(2)$ \\
      & $D_n$ & $D_n$ \\
      & $\mathbb Z_n$ & $\mathbb{Z}_n$ \\
      \midrule
      \multirow{2}{*}{$\mathbb Z_2$} & $D_{2n}$ & $D_n$ \\
      & $\mathbb Z_{2n}$ & $\mathbb Z_n$ \\
      \bottomrule
    \end{tabular}
    \hspace{1em}
    \begin{tabular}{ccc}
      \toprule
      $K/Z_K$ & $K$ & $Z_K$ \\
      \midrule
      \multirow{2}{*}{$D_1$} & $O(2)$ & $SO(2)$ \\
      & $D_n$ & $\mathbb Z_n$ \\
      \midrule
      $\mathbb Z_m(m>2)$ & $\mathbb Z_{mn}$ & $\mathbb Z_n$ \\
      \midrule
      $D_m(m>1)$ & $D_{mn}$ & $\mathbb Z_n$ \\
      \bottomrule
    \end{tabular}
    \caption{Possible $K/Z_K$ in $O(2)$}
  \end{table}\ \\
\vs

\vs

  \subsection{Order of Weyl Group} Our next goal is to effectively compute the order of the Weyl group
  $W_G(S)$, $S \in \mathcal A=\mathcal A(H,K,Z_K)$ (cf.~Algorithm 2). 
%
%
Clearly, the group $C:=N_\Gamma(H)\times(N_{O(2)}(K)\cap N_{O(2)}(Z_K))<G$
acts on $\mathcal I = \mathcal I(H,K,Z_K)$ by
  \begin{align}
    (a,b)_.\chi:=b_.\circ\chi\circ a^{-1}_. \quad (a \in N_\Gamma(H), \; b \in N_{O(2)}(K)\cap N_{O(2)}(Z_K)).
  \end{align}
%
  On the other hand, $C$ also acts on $\mathcal A=\mathcal A(H,K,Z_K)$ by conjugation. It is easy to see that there exists
 a $C$-isomorphism $\mu : \mathcal A \to \mathcal I$. 
 Therefore, one can identify  CCSs in $\mathcal A$ with 
   $C$-orbits in $\mathcal I$.
  Furthermore, given $\chi\in\mathcal I$, its isotropy $C_\chi$ is
  the normalizer of the $S = \mu^{-1}(\chi) \in\mathcal A$.
  By the Orbit-Stabilizer Theorem, one obtains  the following result.
\vs

  \begin{proposition}
   Let  $H<\Gamma$, $Z_K\triangleleft K<O(2)$, $\mathcal A$, $\mathcal I$,  
   $\mu : \mathcal A \to \mathcal I$  and $C<G$ be 
    as above. Take $S\in\mathcal A$ and $\chi := \mu(S) \in\mathcal I$. Then,
    \begin{align}
      \label{eq:order_w}
      \left\lvert W_G(S)\right\rvert=\frac{\left\lvert W_\Gamma(H)\right\rvert
          \left\lvert \left(N_{O(2)}(K)\cap N_{O(2)}(Z_K)\right) / K\right\rvert
          \left\lvert K/Z_K\right\rvert}{\left\lvert C_\chi\right\rvert}.
    \end{align}
  \end{proposition}
\vs
\begin{remark}{\rm 
It should be pointed out that, given a subgroup $S$, the above proposition allows to effectively compute the order of 
$W_G(S)$ 
without computing the normalizer of $S$.} 
\end{remark}

\vs
  \subsection{Example: Conjugacy Classes of Subgroups in $S_4\times O(2)$}
  \label{subsec:ccss_S4xO2}
 Below we show how to identify $\Phi(S_4 \times O(2))$. We start with identifying $\Phi(S_4)$ and $\Phi(O(2))$
 (see Tables \ref{table:phi_s4} and \ref{table:phi_o2}). Next,
 denote by $r$ a rotation generator in $L$ and adopt the modified notation for subgroups
in $S_4\times O(2)$:
  \begin{align*}
    &S=\amall{H}{K}{L}{Z_H}{R},
  \end{align*}
  where $L=K/Z_K$ and $R=\varphi^{-1}(\left<r\right>)$
  ($R$ appears only when $H$, $K$, $Z$ and $L$ are not enough to determine a
  unique conjugacy class of subgroups in $S_4\times O(2)$).
  Finally, using Algorithm 2, one can obtain a complete list
  of conjugacy classes in $S_4 \times O(2)$ (see Table \ref{table:phi_s4_o2} and \cite{DKY}). 
  \begin{table}[h]
    \centering
    \subfigure[CCSs in $S_4$]{
      \label{table:phi_s4}
      \begin{tabular}{rc}
        \toprule
        ID & $(H)$ \\
        \midrule
        1 & $(\mathbb{Z}_1)$ \\
        2 & $(\mathbb{Z}_2)$ \\
        3 & $(D_1)$ \\
        4 & $(\mathbb{Z}_3)$ \\
        5 & $(V_4)$ \\
        6 & $(D_2)$ \\
        7 & $(\mathbb{Z}_4)$ \\
        8 & $(D_3)$ \\
        9 & $(D_4)$ \\
        10 & $(A_4)$ \\
        11 & $(S_4)$ \\
        \bottomrule
      \end{tabular}
    }
    \hskip 8em
    \subfigure[CCSs in $O(2)$]{
      \label{table:phi_o2}
      \begin{tabular}{rc}
        \toprule
        ID & $(K)$ \\
        \midrule
        1 & $(\mathbb{Z}_{n})$ \\
        2 & $(D_{n})$ \\
        3 & $(SO(2))$ \\
        4 & $(O(2))$ \\
        \bottomrule
      \end{tabular}
    }
  \end{table}\ \\
  {\floatsetup[table]{font=scriptsize}
  \begin{table}[h]
    \centering
    \caption{Conjugacy Classes of Subgroups in $S_4\times O(2)$}
    \label{table:phi_s4_o2}
  \scalebox{.9}{  \begin{tabular}{rcc|rcc|rcc}
      \toprule
      ID & $(S)$ & $\left\lvert W(S)\right\rvert$ &
          ID & $(S)$ & $\left\lvert W(S)\right\rvert$ &
          ID & $(S)$ & $\left\lvert W(S)\right\rvert$ \\
      \midrule
      1 & $(\amall{\mathbb{Z}_1}{\mathbb{Z}_{n}}{\mathbb{Z}_{1}}{\mathbb{Z}_1}{})$ & $\infty$ &
          35 & $(\amall{D_4}{D_{n}}{D_{1}}{\mathbb{Z}_4}{})$ & $4$ &
          69 & $(\amall{\mathbb{Z}_2}{SO(2)}{\mathbb{Z}_{1}}{\mathbb{Z}_2}{})$ & $8$ \\
      2 & $(\amall{\mathbb{Z}_2}{\mathbb{Z}_{n}}{\mathbb{Z}_{1}}{\mathbb{Z}_2}{})$ & $\infty$ &
          36 & $(\amall{S_4}{D_{n}}{D_{1}}{A_4}{})$ & $4$ &
          70 & $(\amall{D_1}{SO(2)}{\mathbb{Z}_{1}}{D_1}{})$ & $4$ \\
      3 & $(\amall{D_1}{\mathbb{Z}_{n}}{\mathbb{Z}_{1}}{D_1}{})$ & $\infty$ &
          37 & $(\amall{V_4}{D_{2n}}{D_{2}}{\mathbb{Z}_1}{})$ & $8$ &
          71 & $(\amall{\mathbb{Z}_3}{SO(2)}{\mathbb{Z}_{1}}{\mathbb{Z}_3}{})$ & $4$ \\
      4 & $(\amall{\mathbb{Z}_3}{\mathbb{Z}_{n}}{\mathbb{Z}_{1}}{\mathbb{Z}_3}{})$ & $\infty$ &
          38 & $(\amall{D_2}{D_{2n}}{D_{2}}{\mathbb{Z}_1}{\mathbb{Z}_2})$ & $8$ &
          72 & $(\amall{V_4}{SO(2)}{\mathbb{Z}_{1}}{V_4}{})$ & $12$ \\
      5 & $(\amall{V_4}{\mathbb{Z}_{n}}{\mathbb{Z}_{1}}{V_4}{})$ & $\infty$ &
          39 & $(\amall{D_2}{D_{2n}}{D_{2}}{\mathbb{Z}_1}{D_1})$ & $4$ &
          73 & $(\amall{D_2}{SO(2)}{\mathbb{Z}_{1}}{D_2}{})$ & $4$ \\
      6 & $(\amall{D_2}{\mathbb{Z}_{n}}{\mathbb{Z}_{1}}{D_2}{})$ & $\infty$ &
          40 & $(\amall{D_4}{D_{2n}}{D_{2}}{\mathbb{Z}_2}{V_4})$ & $4$ &
          74 & $(\amall{\mathbb{Z}_4}{SO(2)}{\mathbb{Z}_{1}}{\mathbb{Z}_4}{})$ & $4$ \\
      7 & $(\amall{\mathbb{Z}_4}{\mathbb{Z}_{n}}{\mathbb{Z}_{1}}{\mathbb{Z}_4}{})$ & $\infty$ &
          41 & $(\amall{D_4}{D_{2n}}{D_{2}}{\mathbb{Z}_2}{D_2})$ & $4$ &
          75 & $(\amall{D_3}{SO(2)}{\mathbb{Z}_{1}}{D_3}{})$ & $2$ \\
      8 & $(\amall{D_3}{\mathbb{Z}_{n}}{\mathbb{Z}_{1}}{D_3}{})$ & $\infty$ &
          42 & $(\amall{D_4}{D_{2n}}{D_{2}}{\mathbb{Z}_2}{\mathbb{Z}_4})$ & $4$ &
          76 & $(\amall{D_4}{SO(2)}{\mathbb{Z}_{1}}{D_4}{})$ & $2$ \\
      9 & $(\amall{D_4}{\mathbb{Z}_{n}}{\mathbb{Z}_{1}}{D_4}{})$ & $\infty$ &
          43 & $(\amall{D_3}{D_{3n}}{D_{3}}{\mathbb{Z}_1}{})$ & $2$ &
          77 & $(\amall{A_4}{SO(2)}{\mathbb{Z}_{1}}{A_4}{})$ & $4$ \\
      10 & $(\amall{A_4}{\mathbb{Z}_{n}}{\mathbb{Z}_{1}}{A_4}{})$ & $\infty$ &
          44* & $(\amall{S_4}{D_{3n}}{D_{3}}{V_4}{})$ & $2$ &
          78 & $(\amall{S_4}{SO(2)}{\mathbb{Z}_{1}}{S_4}{})$ & $2$ \\
      11 & $(\amall{S_4}{\mathbb{Z}_{n}}{\mathbb{Z}_{1}}{S_4}{})$ & $\infty$ &
          45* & $(\amall{D_4}{D_{4n}}{D_{4}}{\mathbb{Z}_1}{})$ & $2$ &
          79 & $(\amall{\mathbb{Z}_2}{O(2)}{D_{1}}{\mathbb{Z}_1}{})$ & $8$ \\
      12 & $(\amall{\mathbb{Z}_2}{\mathbb{Z}_{2n}}{\mathbb{Z}_{2}}{\mathbb{Z}_1}{})$ & $\infty$ &
          46 & $(\amall{\mathbb{Z}_1}{D_{n}}{\mathbb{Z}_{1}}{\mathbb{Z}_1}{})$ & $48$ &
          80 & $(\amall{D_1}{O(2)}{D_{1}}{\mathbb{Z}_1}{})$ & $4$ \\
      13 & $(\amall{D_1}{\mathbb{Z}_{2n}}{\mathbb{Z}_{2}}{\mathbb{Z}_1}{})$ & $\infty$ &
          47 & $(\amall{\mathbb{Z}_2}{D_{n}}{\mathbb{Z}_{1}}{\mathbb{Z}_2}{})$ & $8$ &
          81 & $(\amall{V_4}{O(2)}{D_{1}}{\mathbb{Z}_2}{})$ & $4$ \\
      14 & $(\amall{V_4}{\mathbb{Z}_{2n}}{\mathbb{Z}_{2}}{\mathbb{Z}_2}{})$ & $\infty$ &
          48 & $(\amall{D_1}{D_{n}}{\mathbb{Z}_{1}}{D_1}{})$ & $4$ &
          82 & $(\amall{D_2}{O(2)}{D_{1}}{\mathbb{Z}_2}{})$ & $4$ \\
      15 & $(\amall{D_2}{\mathbb{Z}_{2n}}{\mathbb{Z}_{2}}{\mathbb{Z}_2}{})$ & $\infty$ &
          49 & $(\amall{\mathbb{Z}_3}{D_{n}}{\mathbb{Z}_{1}}{\mathbb{Z}_3}{})$ & $4$ &
          83 & $(\amall{D_2}{O(2)}{D_{1}}{D_1}{})$ & $2$ \\
      16 & $(\amall{D_2}{\mathbb{Z}_{2n}}{\mathbb{Z}_{2}}{D_1}{})$ & $\infty$ &
          50 & $(\amall{V_4}{D_{n}}{\mathbb{Z}_{1}}{V_4}{})$ & $12$ &
          84 & $(\amall{\mathbb{Z}_4}{O(2)}{D_{1}}{\mathbb{Z}_2}{})$ & $4$ \\
      17 & $(\amall{\mathbb{Z}_4}{\mathbb{Z}_{2n}}{\mathbb{Z}_{2}}{\mathbb{Z}_2}{})$ & $\infty$ &
          51 & $(\amall{D_2}{D_{n}}{\mathbb{Z}_{1}}{D_2}{})$ & $4$ &
          85 & $(\amall{D_3}{O(2)}{D_{1}}{\mathbb{Z}_3}{})$ & $2$ \\
      18 & $(\amall{D_3}{\mathbb{Z}_{2n}}{\mathbb{Z}_{2}}{\mathbb{Z}_3}{})$ & $\infty$ &
          52 & $(\amall{\mathbb{Z}_4}{D_{n}}{\mathbb{Z}_{1}}{\mathbb{Z}_4}{})$ & $4$ &
          86 & $(\amall{D_4}{O(2)}{D_{1}}{V_4}{})$ & $2$ \\
      19 & $(\amall{D_4}{\mathbb{Z}_{2n}}{\mathbb{Z}_{2}}{V_4}{})$ & $\infty$ &
          53 & $(\amall{D_3}{D_{n}}{\mathbb{Z}_{1}}{D_3}{})$ & $2$ &
          87 & $(\amall{D_4}{O(2)}{D_{1}}{D_2}{})$ & $2$ \\
      20 & $(\amall{D_4}{\mathbb{Z}_{2n}}{\mathbb{Z}_{2}}{D_2}{})$ & $\infty$ &
          54 & $(\amall{D_4}{D_{n}}{\mathbb{Z}_{1}}{D_4}{})$ & $2$ &
          88 & $(\amall{D_4}{O(2)}{D_{1}}{\mathbb{Z}_4}{})$ & $2$ \\
      21 & $(\amall{D_4}{\mathbb{Z}_{2n}}{\mathbb{Z}_{2}}{\mathbb{Z}_4}{})$ & $\infty$ &
          55 & $(\amall{A_4}{D_{n}}{\mathbb{Z}_{1}}{A_4}{})$ & $4$ &
          89 & $(\amall{S_4}{O(2)}{D_{1}}{A_4}{})$ & $2$ \\
      22 & $(\amall{S_4}{\mathbb{Z}_{2n}}{\mathbb{Z}_{2}}{A_4}{})$ & $\infty$ &
          56* & $(\amall{S_4}{D_{n}}{\mathbb{Z}_{1}}{S_4}{})$ & $2$ &
          90 & $(\amall{\mathbb{Z}_1}{O(2)}{\mathbb{Z}_{1}}{\mathbb{Z}_1}{})$ & $24$ \\
      23 & $(\amall{\mathbb{Z}_3}{\mathbb{Z}_{3n}}{\mathbb{Z}_{3}}{\mathbb{Z}_1}{})$ & $\infty$ &
          57 & $(\amall{\mathbb{Z}_2}{D_{2n}}{\mathbb{Z}_{2}}{\mathbb{Z}_1}{})$ & $8$ &
          91 & $(\amall{\mathbb{Z}_2}{O(2)}{\mathbb{Z}_{1}}{\mathbb{Z}_2}{})$ & $4$ \\
      24 & $(\amall{A_4}{\mathbb{Z}_{3n}}{\mathbb{Z}_{3}}{V_4}{})$ & $\infty$ &
          58 & $(\amall{D_1}{D_{2n}}{\mathbb{Z}_{2}}{\mathbb{Z}_1}{})$ & $4$ &
          92 & $(\amall{D_1}{O(2)}{\mathbb{Z}_{1}}{D_1}{})$ & $2$ \\
      25 & $(\amall{\mathbb{Z}_4}{\mathbb{Z}_{4n}}{\mathbb{Z}_{4}}{\mathbb{Z}_1}{})$ & $\infty$ &
          59 & $(\amall{V_4}{D_{2n}}{\mathbb{Z}_{2}}{\mathbb{Z}_2}{})$ & $4$ &
          93 & $(\amall{\mathbb{Z}_3}{O(2)}{\mathbb{Z}_{1}}{\mathbb{Z}_3}{})$ & $2$ \\
      26 & $(\amall{\mathbb{Z}_2}{D_{n}}{D_{1}}{\mathbb{Z}_1}{})$ & $16$ &
          60 & $(\amall{D_2}{D_{2n}}{\mathbb{Z}_{2}}{\mathbb{Z}_2}{})$ & $4$ &
          94 & $(\amall{V_4}{O(2)}{\mathbb{Z}_{1}}{V_4}{})$ & $6$ \\
      27 & $(\amall{D_1}{D_{n}}{D_{1}}{\mathbb{Z}_1}{})$ & $8$ &
          61* & $(\amall{D_2}{D_{2n}}{\mathbb{Z}_{2}}{D_1}{})$ & $2$ &
          95 & $(\amall{D_2}{O(2)}{\mathbb{Z}_{1}}{D_2}{})$ & $2$ \\
      28 & $(\amall{V_4}{D_{n}}{D_{1}}{\mathbb{Z}_2}{})$ & $8$ &
          62 & $(\amall{\mathbb{Z}_4}{D_{2n}}{\mathbb{Z}_{2}}{\mathbb{Z}_2}{})$ & $4$ &
          96 & $(\amall{\mathbb{Z}_4}{O(2)}{\mathbb{Z}_{1}}{\mathbb{Z}_4}{})$ & $2$ \\
      29 & $(\amall{D_2}{D_{n}}{D_{1}}{\mathbb{Z}_2}{})$ & $8$ &
          63 & $(\amall{D_3}{D_{2n}}{\mathbb{Z}_{2}}{\mathbb{Z}_3}{})$ & $2$ &
          97 & $(\amall{D_3}{O(2)}{\mathbb{Z}_{1}}{D_3}{})$ & $1$ \\
      30 & $(\amall{D_2}{D_{n}}{D_{1}}{D_1}{})$ & $4$ &
          64 & $(\amall{D_4}{D_{2n}}{\mathbb{Z}_{2}}{V_4}{})$ & $2$ &
          98 & $(\amall{D_4}{O(2)}{\mathbb{Z}_{1}}{D_4}{})$ & $1$ \\
      31 & $(\amall{\mathbb{Z}_4}{D_{n}}{D_{1}}{\mathbb{Z}_2}{})$ & $8$ &
          65* & $(\amall{D_4}{D_{2n}}{\mathbb{Z}_{2}}{D_2}{})$ & $2$ &
          99 & $(\amall{A_4}{O(2)}{\mathbb{Z}_{1}}{A_4}{})$ & $2$ \\
      32 & $(\amall{D_3}{D_{n}}{D_{1}}{\mathbb{Z}_3}{})$ & $4$ &
          66* & $(\amall{D_4}{D_{2n}}{\mathbb{Z}_{2}}{\mathbb{Z}_4}{})$ & $2$ &
          100 & $(\amall{S_4}{O(2)}{\mathbb{Z}_{1}}{S_4}{})$ & $1$ \\
      33 & $(\amall{D_4}{D_{n}}{D_{1}}{V_4}{})$ & $4$ &
          67* & $(\amall{S_4}{D_{2n}}{\mathbb{Z}_{2}}{A_4}{})$ & $2$ \\
      34 & $(\amall{D_4}{D_{n}}{D_{1}}{D_2}{})$ & $4$ &
          68 & $(\amall{\mathbb{Z}_1}{SO(2)}{\mathbb{Z}_{1}}{\mathbb{Z}_1}{})$ & $48$ \\
      \bottomrule
    \end{tabular}}
    \label{table:ccsstbl_S4xO2}
  \end{table}
  }
\vs
 \section{Computation in $A(\Gamma\times O(2))$}
  \label{appendix:b}
%
  Given an (infinite) group $G = \Gamma \times O(2)$, 
  it is a difficult task,
  in general, to perform the multiplication in $A(G)$ according to formula \eqref{eq:Burnside-geom}.
  In practice, given generators $(S_1), (S_2) \in A(G)$,  one can use the following  {\em Recurrence Formula}:
  \begin{align}
    \label{eq:burnside_mul_summation}
    (S_1) \cdot (S_2)=\sum_{(S) \in\Phi_0(G)}{n_S(S)},
  \end{align}
  where
  \begin{align}
    \label{eq:recurrence_formula}
    n_S=\frac{n_G(S,S_1)\left\lvert W_G(S_1)\right\rvert
        n_G(S,S_2)\left\lvert W_G(S_2)\right\rvert-\sum_{(\tilde S)>(S)}
        {n_G(S,\tilde S)n_{\tilde S}\left\lvert W_G(\tilde S)\right\rvert}}
        {\left\lvert W_G(S)\right\rvert}
  \end{align}
  (cf.~Subsection \ref{subsec:jargon} and \cite{AED}).
  \begin{remark}\label{rmk:psi_0}\rm
    In \eqref{eq:recurrence_formula}, $n_S=0$ for
    $(S)\notin\Psi_0(S_1,S_2)$, where
    \begin{align*}
      \Psi_0(S_1,S_2)=\left\{(S)\in\Phi_0(G):(S)<(S_1),(S_2)\wedge\dim S=\min\left\{\dim S_1,\dim S_2\right\}\right\}.
    \end{align*}
  \end{remark}
  \vs
  Despite the algebraic nature of formula \eqref{eq:recurrence_formula}, the computation of the numbers 
  $n(S,S_1), n(S,S_2)$ and $n(S,\tilde S)$ cannot be implemented in GAP directly due to the infiniteness 
  of $G$. To circumvent this obstacle, one can replace $O(2)$ by $D_m$ with $m$ large enough. To this end,
  given $(S_1),(S_2)\in\Phi_0(G)$, one may consider a map $\eta:\Psi_1(S_1,S_2)\to\Phi_0(G')$,
  where $G' =\Gamma \times D_m$ and
  \begin{align*}
    \Psi_1(S_1,S_2)=\Psi_0(S_1,S_2)\cup\{(S_1),(S_2)\}
  \end{align*}
  (cf.~Remark \ref{rmk:psi_0}) such that
  \begin{align}
    \label{eq:circumvent_inf}
    n_G(P,Q)=n_{G'}(\eta(P),\eta(Q)).
  \end{align}
  Observe that the choice of $D_m$ and $\eta$ essentially depends on $(S_1)$ and $(S_2)$. To be more specific,  
  suppose $S_1,S_2\in G$ are given by
  \begin{align}
    S_1&=\amal{H_1}{K_1}{L_1}{\varphi_1}{\psi_1}\\
    S_2&=\amal{H_2}{K_2}{L_2}{\varphi_2}{\psi_2},
  \end{align}
  and consider three different cases:
  \renewcommand{\labelenumi}{(\roman{enumi})}
  \begin{enumerate}
  \item $\left\lvert K_1\right\rvert<\infty$ and $\left\lvert K_2\right\rvert<\infty$.
    \smallskip\noindent
    In this case, $K_1=D_j$ and $K_2=D_k$ for some $j$ and $k$. Take $m=2\gcd(j,k)$. Then,
    there is a natural embedding $\tau:S_1\cup S_2\rightarrow G'$.
    Given $(S)\in\Psi_1(S_1,S_2)$, one may assume that $S\subset S_1\cup S_2$ and
    define $\eta((S))=(\tau(S))$.
  \item $\left\lvert K_1\right\rvert<\infty$ and $\left\lvert K_2\right\rvert=\infty$.
    \smallskip\noindent
    In this case, $K_1=D_j$ and $K_2=O(2)$ or $SO(2)$. Take $m=2j$ and
    put 
    \begin{align}
      \label{eq:associate_o2_and_dm}
      \begin{cases}
        \rho(O(2))=D_m\\
        \rho(SO(2))=\mathbb Z_m
      \end{cases}.
    \end{align}
    Also, there is a natural embedding $\tau:S_1\to G'$. Then, given $(S)\in\Psi_1(S_1,S_2)$, define
    \begin{align*}
      \eta((S))=\begin{cases}
        (\amal{H_2}{\rho(K_2)}{L}{\varphi_2}{\tilde\psi_2}),&\mbox{if }(S)=(S_2),
        \quad(\ker(\tilde\psi_2)=\rho(\ker(\psi_2)))\\
        (\tau(T)),&\mbox{otherwise}
      \end{cases}.
    \end{align*}
  \item $\left\lvert K_1\right\rvert=\infty$ and $\left\lvert K_2\right\rvert=\infty$.
    \smallskip\noindent
    In this case, $K_1$ and $K_2$ are $O(2)$ or $SO(2)$. Take $m=1$.
    Given $(S)=(\amal{H}{K}{L}{\varphi}{\psi})\in\Psi_1(S_1,S_2)$, define
    \begin{align*}
      \eta((S))=(\amal{H}{\rho(K)}{L}{\varphi}{\tilde\psi})\quad(\ker(\tilde\psi)=\rho(\ker(\psi)))
    \end{align*}
    (cf.~\eqref{eq:associate_o2_and_dm}).
  \end{enumerate}
%

\vs

\end{document}